\documentclass[11pt,twoside,a4paper]{amsart}
\usepackage{amsmath,amssymb,amsthm}

\usepackage{graphicx}  %psfrag
\newtheorem{Thm}{Theorem}[section]
\newtheorem{Lem}[Thm]{Lemma}
\newtheorem{Pro}[Thm]{Proposition}
\newtheorem{Cor}[Thm]{Corollary}

\theoremstyle{definition}

\theoremstyle{remark}
\newtheorem{Rem}[Thm]{Remark}
\newtheorem{Def}[Thm]{Definition}

\newcommand{\R}{\mathbb{R}}

\newcommand{\Hy}{\mathbb{H}}

\newcommand{\cF}{\mathcal{F}}
\newcommand{\cH}{\mathcal{H}}

\newcommand{\cM}{\mathcal{M}}

\newcommand{\cW}{\mathcal{W}}

\renewcommand{\phi}{\varphi}

\newcommand{\id}{\operatorname{id}}

\newcommand{\Vol}{\operatorname{Vol}}

\newcommand{\tr}{\operatorname{tr}}

\renewcommand{\cos}{\operatorname{cos}}

\begin{document}

\title  [Asymptotically harmonic manifolds  with  $h = 0$] {Geometry of Asymptotically harmonic manifolds with minimal horospheres}
\author {Hemangi Shah*}
\thanks{*The author  would like to  thank University of Extremadura, Spain, for  its  hospitality and  support while this  work  was initiated}
\dedicatory{Hemangi  Shah  would  like  to  dedicate  this article  to  her  father, Shri Madhusudan  Shah, to  ignite a spark 
of creativity  in  her.}
\address{Harish Chandra Research Institute,  HBNI, Chhatnag Road, Jhusi, Allahabad 211019, India.}
\email{hemangimshah@hri.res.in}

 \begin{abstract}
\noindent 
 $(M^n,g)$ be a complete  Riemannian manifold  without conjugate points.  
 In  this  paper,  we  show  that  if  $M$ is also simply connected,  then 
  $M$  is flat,  provided  that  $M$  is also asymptotically harmonic  manifold  with minimal  horospheres  (AHM). 
 The (first order)  flatness of  $M$  is  shown  by  using the strongest  criterion: 
  $\{{e_i}\}$  be  an orthonormal  basis  of  $T_{p}M$  and  $\{b_{e_{i}}\}$ 
 be  the corresponding  Busemann  functions  on $M$.   Then,  
(1) The  vector  space  $V =  span\{b_{v} | v \in T_{p}M \}$
is  finite  dimensional  and  dim  $V  = $  dim $M = n$.
(2)  $\{\nabla b_{e_i}(p) \}$  is  a  global  parallel  orthonormal basis  of  $T_{p}M$  
for  any  $p \in  M$.   Thus,  $M$  is  a  parallizable  manifold.
And
(3)$F  :  M  \rightarrow  {\R}^n$ defined by
$F(x)  =  (b_{e_1}(x),  b_{e_{2}}(x),  \cdots, b_{e_{n}}(x)),$
is  an  isometry   and  therefore, $M$  is  flat.  Consequently,  AH manifolds can have  either  polynomial or  exponential volume growth,
generalizing  the corresponding result  of \cite{N.05}  for  harmonic  manifolds. In  case  of  harmonic  manifold  with minimal  horospheres  (HM),  the  (second order)  flatness was  proved in  \cite{RS.02b}  by showing that  $span\{b_{v}^2 | v \in T_{p}M \}$  is  finite  dimensional. We conclude that, the  results  obtained  in this  paper  are the strongest and  wider  in  comparison  to harmonic manifolds,  which are  known  to be AH. \\
In fact, our proof shows the more generalized result, viz.:
If $(M,g)$ is a non-compact, complete, connected Riemannian manifold of infinite injectivity radius and of subexponential volume growth, then M is a first order flat manifold.   
\end{abstract}
\keywords{Asymptotically harmonic manifold, harmonic manifold, Busemann function, mean curvature, horospheres, minimal horospheres, volume  growth entropy,
 Killing  vector field} 
\subjclass[2010]{Primary 53C20; Secondary 53C25, 53C35}
\maketitle 

\tableofcontents

\section{Introduction and Preliminaries}
\indent 
One  of  the important problems  in the  geometry  of  Hadamard  manifold (manifold of nonpositive curvature) $M$  is:  To  what  extent  
{\it horospheres} (geodesic sphere  of  infinite radius) of  $M$  determine the  geometry  of  $M$? This  is  a  very  popular topic of current 
research as  can be  seen from the  very  recent  paper  \cite{IKPS.17},  where  the connection between hessian of  Busemann functions
and  rank  of  Hadamard  manifold  is  established. \\\\
\indent  
Horospheres,  geodesic spheres  of  infinite radius,  of  $M$, are level  sets  of  {\it Busemann functions}, which are  defined  below.\\
Let $M$  be  a  complete, simply connected Riemannian manifold without conjugate points.  Then  by  Cartan-Hadamard  theorem,
every  geodesic  of  $M$  is  a line.   Let $SM$ be the unit tangent bundle of $M$.
For $v \in SM$, let $\gamma_{v}$ be the geodesic line in $M$ with ${\gamma'_{v}}(0) = v$. Then
 $b{_v}^{+}:M\rightarrow \R$  and  $b{_v}^{-}:M\rightarrow \R$  denote the two  Busemann functions
associated to  $\gamma_v$,  respectively,  towards  $+\infty$  and  $-\infty$,  and are defined as:
\begin{eqnarray*}
      b_{v}^{+}(x) &=& \lim_{t\rightarrow\infty} d(x,\gamma_v(t))- t,\\
      b_{v}^{-}(x) &=& \lim_{t\rightarrow-\infty} d(x,\gamma_v(t))+ t.
\end{eqnarray*}
 Note that $b_{v}^{\pm}(\gamma_v(r))=\mp r$. The level sets, ${{b_v}^{\pm}}^{-1}(t)$, are called {\it horospheres} of $M$.
Thus, the two Busemann function  can be interpreted as distance  function from $\pm \infty$,  can  be  defined  for  any  line  in  $M$.
 Refer \cite{P.98} for details on Busemann functions.  \\
 
 A complete Riemannian manifold is called \emph{harmonic}, if all the geodesic spheres of sufficiently small radii
 are of constant mean curvature. The known examples of harmonic manifolds include flat spaces and locally rank one symmetric spaces.
 In $1944$, Lichnerowicz conjectured that any simply connected harmonic manifold is either flat or a rank one symmetric space. 
 This conjecture is true in  compact case and false in  non-compact one.
  However, the other questions in the non-compact case remain open. 
  For the development of the conjecture see the references in \cite{RS.02b}.
   In particular, if an harmonic manifold does not have conjugate points, then the geodesic spheres of 
   any radii are of constant mean curvature and is also AH (cf. \cite[Remark 2.2]{H.06},  cf. \cite{RS.03}). 
   Moreover, by  Allamigeon's Theorem \cite{B.78},  it follows that any complete, simply connected and non-compact harmonic manifold has 
   no conjugate points.  For definition and more details about harmonic manifolds see  \cite{B.78}.  In  \cite{H.06} J.~Heber 
   obtained the complete classification of simply connected harmonic manifolds in the homogeneous case. \\\\
   
 \indent
It  is  very  well  known  that, eg.,  real  hyperbolic  space, ${\Hy}^n$,  real  space  form,  is  characterized  by  its  horospheres.
Equivalently,  ${\Hy}^n$,   which  is  a {\it  harmonic  manifold},  is  characterized  by  its  volume  density  function,  denoted  by $\Theta(r)$, which
in  polar co-ordinates  can be  expressed  as $\Theta(r) =  r^{n-1}   \sqrt  {\det {g_{ij}}}(r)$.  In this case, 
$\Theta(r) =  {\sinh}^{n-1}(r)$ 
and the  mean curvature of horospheres  $\displaystyle \frac{\Theta'(r)}{\Theta(r)} =  (n-1) =  \Delta {b_r}^{\pm}$. \\
 Szabo  asked:  To  what extend  the  density function  of a  harmonic  manifold $M$,  determine  the  geometry  of  $M$?
 The  affirmative  answer  to  his  question  can be  found  in  \cite{RR.97},  in  case  of harmonic symmetric  spaces of rank one.   \\
 
 {\it Asymptotically harmonic}  manifolds 
 are  asymptotic  generalization  of harmonic  manifolds.  
 They  were   originally  introduced   by  Ledrappier (\cite{L.90},  Theorem  $1$),
in connection  with   rigidity  of measures  related  to  the   Dirichlet
problem  (harmonic  measure)  and  the  dynamics  of  the geodesic  flow
(Bowen-Margulis  measure).

\begin{Def}\label{ah}
 A complete, simply connected Riemannian manifold without conjugate
points is called  AH, if the mean
curvature of its horospheres is a universal constant. Equivalently, if
its Busemann functions satisfy $\Delta {b_v}^{\pm}  \equiv h,\; \forall v\in SM$,  where  $h$ is a nonnegative constant. 
\end{Def}
Then by regularity of elliptic partial differential equations, ${b_v}^{\pm}$ is a smooth function on $M$ for all $v$ and all horospheres of $M$
are smooth, simply connected hypersurfaces in $M$ with constant mean curvature $h\geq 0$. 
For example, every simply connected, complete harmonic manifold without
conjugate points is   AH.  See  \cite{RS.03}  for  a proof.\\
   
  \indent
  Till  to-date  the  only known examples of  AH manifolds  are   harmonic manifolds.  
  Intuitively one  thinks  that  the  class  of  harmonic  and  AH manifolds  coincide.  And the Lichnerowicz
   conjecture  should  be  true for  AH manifolds  as  well.  In this  paper  we  answer  this  question partially
   that  these  two  classes coincide  in  case  when $h =0$.   For  more  details on questions related to Lichnerowicz  conjecture
   for AH  manifolds  see \cite{Z.11}, \cite{Z.12}  and  \cite{KP.05}. \\

Note  that  study  of AH manifolds starts  from  dimension $2$.
 In dimension 2, the only AH manifolds are  symmetric  viz., $\Hy^2$ and $\R^2$.   
 In dimension 3, AH manifolds have been classified in \cite{HKS.07,K.94,SS.08,S.16} and in non-compact, Einstein and homogeneous case in \cite{H.06}. For more details on AH spaces, we refer to the
discussion and to the references in \cite{HKS.07}. Important results in this context are contained in \cite{BCG.95} and \cite{K.94}.\\

Thus, AH manifolds are defined  via  the  property that all  of  its  horospheres,
a  parallel family of hypersurfaces, have the  same  constant mean curvature.  In view  of the  above  classification
the  natural analogue  of  Szabo's  question is:
whether the  asymptotic  growth of volume density function or equivalently  whether the geometry of  horospheres of an AH manifold influence  geometry  of  its  ambient  space?\\

For $v\in SM$, the corresponding horosphere is {\it totally umbilical} with constant principal curvature 
$h \in \R$ if and only if $u^+(t)=h \id_{\gamma_v'(t)^\bot}$. 
Thus, in  this  case  $M$ is Einstein of constant non-positive curvature  $-h^2$, 
as   $R(t)=-h^2 id_{\gamma_v'(t)^\bot}$ from \eqref{eq:ricatti}. 
Thus,  if  $h =0$, then  $u^+(t) \equiv 0$, i.e. if every horosphere of an AH  manifold is totally geodesic, 
 then the ambient manifold is flat.  See \cite{ISS.14} for more rigidity results  about  AH manifolds.\\
Moreover, the classification of $3$-dimensional AH manifolds  follows  from  \cite{SS.08}  and the recent paper
 \cite{S.16}.   In particular,  in \cite{S.16}  the following  result  was  proved:

\begin{Thm}\label{flat3}
 $(M, g)$ be a complete and simply connected Riemannian manifold of dimension $3$ without conjugate points. If $M$ is  an AH of constant $h = 0$, then $M$ is flat.
  \end{Thm} 

In case of harmonic manifolds flatness follows in all
dimensions from  \cite{RS.02b}.

\begin{Thm}\label{flat1}  \cite{RS.02b}
  $(M, g)$ be a complete and simply connected Riemannian manifold. 
 If $M$ is a  non-compact harmonic  manifold  of polynomial volume  growth, then $M$ is flat. 
 \end{Thm} 

\begin{Cor} \label{growth}  \cite{N.05}
 A  non-compact harmonic  manifold  either  has polynomial  volume  growth  
or  of  exponential  volume  growth.  Consequently,  if  $M$  is  HM,
then  $M$  is  flat.   
\end{Cor}

In  case  of HM,  it  was  shown  that    \cite{RS.02b},  the  vector  spaces  $V =  span\{b_{v} | v  \in S_{p}M \}$  
and  $W =  span\{b_{v}^2 | v \in S_{p}M \}$   are  finite  dimensional,  where  $b_v = b_{v}^{+}$ is  Busemann function for $v$. 
Then  averaging  ${b_v}^2$   (idea which can be employed only for harmonic manifolds), 
 it  was  shown  that  HM  is  Ricci  flat  and  hence  flat.  Thus,  the  flatness  Theorem  \ref{flat1}  of \cite{RS.02b}  was  proved
by  using  finite  dimensionality  of  $W$.  Hence,  we  term  this  as {\it  second order flatness}.  See  Theorem  \ref{sketch}
 of  $\S 6$  for  a   sketch  of  proof  of  Theorem  \ref{flat1}  of   \cite{RS.02b}.
 Note  that to prove flatness of  HM, in  \cite{RS.02b},  the  natural  fact  that  $V$  is finite dimensional  was  not explored.  \\
 
 In  this  paper, we  show  that in an AHM  $(M^n,g)$, dim $V = n$ and  consequently, $M$  is  flat.
  Thus,  proving  the flatness  of  AHM  by  the  strongest  flatness criterion and  also  strengthening  the  flatness result,
 Theorem \ref{flat1} of  \cite{RS.02b}.   We  term  this  as {\it  first order flatness}  of $M$. \\
 
 The  main  result  of  our  paper  is  Theorem \ref{f}.
 
\begin{Thm}\label{f}
 $(M^n,g)$ be  an AHM  with $\{{e_i}\}$  an orthonormal  basis  of  $T_{p}M$  and   $\{b_{e_{i}}\}$,
 the corresponding  Busemann  functions  on $M$.   Then
 \begin{itemize}
\item[(1)] The  vector  space  $V =  span\{b_{v} | v \in T_{p}M \}$
is  finite  dimensional  and  dim  $V  = $  dim $M = n$.   
\item[(2)]  $\{\nabla b_{e_i}(p) \}$  is  a  global  parallel  orthonormal basis  of  $T_{p}M$  
for  any  $p \in  M$.   Thus,  $M$  is  a  parallizable  manifold. 
\item[(3)]
$F  :  M  \rightarrow  {\R}^n$ defined by
$F(x)  =  (b_{e_1}(x),  b_{e_{2}}(x),  \cdots, b_{e_{n}}(x)),$
is  an  isometry   and  therefore, $M$  is  flat. 
\end{itemize}
\end{Thm}    
 
 One  of the important  steps in  proving  main  Theorem  \ref{f}  is,
Strong Liouville Type Property, that is any  subharmonic  function  bounded  from above  is  constant on an AHM, is  proved  in $\S 3$.  
In fact, in $\S 3$ we  show the more generalized result that, if $(M,g)$ is a
non-compact, complete, connected manifold of infinite injectivity radius and of subexponential volume growth, then $(M,g)$ satisfies Strong Liouville Type Property, and hence Liouville Property. We also show that it satisfies $L^1$ Liouville Property too.
To prove Strong Liouville Type Property, we derive an integral formula for the derivative of mean value of a $C^1$ function. See $\S 3$ for more details.\\    
Using  results  of  $\S 3$, the  strongest  result of the paper, that every non-trivial  Killing field  on  AHM  is  parallel, is proved in $\S 4$.  
In  $\S 5$,  it  is  established  that  any  parallel unit vector  field on  AHM is  Killing of  the  form  $\nabla b_{v},$   $v \in  SM$.
This  in  turn  proves  Theorem  \ref{f}  in $\S 5$. 
$\S 6$ is the concluding section of the paper, where  we  note  all  consequences  of  flatness. We also conclude the final more generalized result of this paper viz. :
\begin{Thm}\label{ffs}
If $(M,g)$ is a non-compact, complete, connected manifold of infinite injectivity
radius and subexponential volume growth, then $(M,g)$ is first order flat. 
\end{Thm}
 \noindent
 The proof  of  the result that, if
$M$  is  an AHM  admitting  compact  quotient,  then 
$M$  must be flat, is included in  Appendix, $\S 7$, the
last section of  this  paper. We begin by analyzing Busemann functions on AHM's in
$\S2$.\\

\indent
Finally,  we  note that as any non-compact harmonic manifold is AH  \cite{RS.03},  we recover Theorem  \ref{flat1}
from  our main Theorem \ref{f}. \\

We  now describe  general  preliminaries  on  an  AHM  which  will be  used  throughout  the paper. \\
\indent
In  the  sequel,  the  term  AH manifold $M$  tacitly means simply  connected,  complete  Riemannian  manifold
without  conjugate  points.  Therefore, ${b_v}^{\pm}$ are smooth functions on $M$ for all $v$ and all horospheres of $M$
are smooth, simply connected hypersurfaces in $M$ with constant mean curvature $h\geq 0$.\\
Hence, on  an   AH   manifold  $M$ we  can define  $(1,1)$  tensor  fields $u^{+}$ and  $u^{-}$   as  follows: 
For $v \in SM$ and $x \in v^{\perp},$ let
$$ u^{+}(v)(x) = \nabla_x \nabla b_{-v}^{+} \ \ \ \mbox{and}\ \ \
u^{-}(v)(x) = - \nabla_x \nabla b_{v}^{+}.$$ Thus, $u^{\pm}(v) \in
\mbox{End} \;(v^{\perp})$ and $\tr u^{+}(v)  =  \Delta{ b^+_{-v}}= h$,
$\tr u^{-}(v)  = - \Delta{ b^+_{v}} = - h$.  $u^{+}(v),  u^{-}(v)$,  respectively, is the second fundamental form of the
{\it unstable and stable horospheres} ${{b^-_v}}^{-1}(-c),  {{b^+  _v}}^{-1}(c)$, $c \in  \R$, respectively, at $p$.  \\
Moreover, the endomorphism fields $u^{\pm}$ satisfy the Riccati equation along the orbits
of the geodesic flow ${\phi}^{t} : SM \rightarrow SM$.
Thus, if $u^{\pm}(t):=u^{\pm}(\phi^tv)$ and   the  Jacobi  operator,
$R(t):=R(\cdot,\gamma_v'(t))\gamma_v'(t)\in\mbox{End}(\gamma_v'(t)^{\perp}),$
then
\begin{equation}\label{eq:ricatti}
(u^{\pm})'(t) + (u^{\pm})^2(t) + R(t) = 0.
\end{equation} $u^{+}(t)$  and  $u^{-}(t)$   are  called  as  {\it unstable  and  stable}  Riccati solution respectively.\\

Let $(M,g)$ be a complete Riemannian manifold.
Let $\Theta_{x}(u, r)$ be the density of the volume form of
$M$ in normal coordinates centered at some point $x$; so the volume form reads
${dV}_{M} = \Theta_{x}(u, r){dV}_{S_{x}M}dr$ where ${dV}_{S_{x}M}$ is the volume form of 
${S_{x}M}$, where ${dV}_{S_{x}M}$ is the volume form of ${S_{x}M}$.\\\\
\noindent
The function $\Theta_{x}$ is related to the mean curvature $h_x$ of spheres centered at
$x$ and radius $r$ by the formula 

\begin{eqnarray}\label{hor}
\frac{\Theta_{x}'(u, r)}{\Theta_{x}(u, r)} = h_{x}(exp_{x} (ru)),
\end{eqnarray} 
where ${\Theta_{x}'(u, r)}$ denotes the derivative of ${\Theta_{x}(u, r)}$ with respect to r.
In what follows, we shall often write for short the point $\exp_{x}(ru)$ as $(u, r)$
to avoid cumbersome notations; moreover, we will regard ${\Theta_{x}(u, r)}$ as a function on $SM \times \R$.\\\\
If $\gamma_{v}$ is a  geodesic in $M$ with $\gamma_{v}(0) =p, \gamma_{v}'(0)=v,$
in an AH, then it follows that  
$\Delta b_{v}^{\pm}= \displaystyle\lim_{r \rightarrow \infty} \frac{\Theta_{p}'(v,r)}{ \Theta_{p}(v,r)} = const = h \geq 0$,  mean curvature of  horospheres of $M$,
viz., ${b_{v}^{\pm}}^{-1}(t), t \in \R$.

\begin{Pro}\label{non-pos}
If  $(M,g)$  is  an  AHM,  then  $Ricci_{M}  \leq  0$.
\end{Pro}
\begin{proof}
From  Riccati  equation we get,
\begin{equation}\label{eq:ricatti}
(u^{\pm})'(t) + (u^{\pm})^2(t) + R(t) = 0.
\end{equation} 
 Hence,  $\tr(u^{+})^2 +\tr R = 0.$
 This  implies  that
 $Ricci(v,v)  \leq  0, \; \mbox{for  any} \; v  \in  SM.$
 \end{proof}
 
\section{Busemann functions on AHM's}
In this section we describe the behaviour of Busemann functions on AHM's. 
We show that on an  AHM, {\it asymptotic} geodesics \cite{P.98} are{ \it bi-asymptotic}.
Using Liouville Property (Corollary \ref{Lio1}, proved in $\S 3$), we conclude that any two geodesics of AHM having finite Hausdorff distance towards $+\infty$ are bi-asymptotic. Using standard techniques, we prove that in an AHM there exists bounded strip in all directions. We also describe ideal boundary of an AH manifolds. \\

If $(M,g)$ is an AH, then by Definition \ref{ah} all of its Busemann functions satisfy 
$\Delta {b_v}^{\pm}  \equiv h,\; \forall v\in SM$,  where  $h$ is a nonnegative constant. 

\begin{Pro}\label{smooth}
If $(M,g)$ is an AH manifold, then all of its Busemann functions are smooth. 
In particular, if $(M,g)$ is a real analytic manifold, then Busemann functions are
real analytic. Consequently, if $(M,g)$ is a harmonic manifold, then all of its Busemann functions are real analytic.
\end{Pro}
\begin{proof}
We have by definition $\Delta {b_v}^{\pm}  \equiv h,\; \forall v\in SM$.
And Laplacian on $M$ is an elliptic operator with smooth coefficients. 
Then by regularity of elliptic partial differential equations, viz.,
H\"ormander's Theorem $7.5.1$ of \cite{H.69}, pg. $178$, we conclude that
${b_v}^{\pm}$ is a smooth function on $M$ for all $v$.\\
The conclusion for harmonic manifold follows, as they are analytic and AH \cite{RS.03}.
\end{proof}

Now we study  ideal boundary of an AH manifold. We begin with the definition of 
asymptote on AH manifolds.    

\begin{Def}
\begin{itemize}
\item[1)]{Asymptote :} Let $\gamma_v$ be a ray  in  a complete, non-compact 
Riemannian manifold $(M,g)$.  A vector
$w \in S_pM, p \in M$ is called an
asymptotic to $v$, if $\nabla b_{v,s_i}(p) \rightarrow w$ for 
some  sequence $s_i \rightarrow \infty$. Clearly, $\gamma_w$ is also 
a ray. We also say that $\gamma_w$ is 
asymptotic ray to $\gamma_v$ from $p$  towards  $+ \infty$. 
\item[2)]{ Bi-asymptote :} We say that two geodesics are
bi-asymptotic, if they are asymptotic to each other in both
 positive and negative directions. 
\item[3)] If all the asymptotic geodesics of $M$ are bi-asymptotic,
then we say that $M$ has  {\em bi-asymptotic property}. 
\end{itemize}
\end{Def}

The next proposition is an important result proved
in \cite{P.98}.

\begin{Pro}
  $(M,g)$ be a complete, non-compact 
Riemannian manifold. If ${\gamma_w}$ is an asymptote to
 $\gamma_v$ from $p$, then their
{\em Busemann functions} are related by

\begin{eqnarray}
 b_{v}(x)-b_{w}(x) \leq b_{v}(p).
 \label{e8}
\end{eqnarray}
\begin{eqnarray}
 b_{v}({\gamma_w}(t))& =& b_{v}(p) +
{b}_{w}({\gamma_w}(t)), \label{e9}\\
& =& b_{v}(p) - t.
\label{e10}
\end{eqnarray}
\end{Pro}

 Proof  of  Corollary  \ref{b1},  \ref{b2} and  \ref{b3}  which follows  for  AH manifolds,  
is  on  the  same  lines  as the  proof in case of harmonic manifolds  \cite{RS.03}. 

\begin{Cor}\label{b1}
 If $\gamma_v$ is a geodesic, then 
asymptote to $\gamma_v$ through $p \in M,$ AH,  is unique. 
\label{c8}
\end{Cor}
\begin{proof}
Let ${\gamma_w}$  be an asymptote to $\gamma_v$ from $p$.\\
Then,
$$ \left<\nabla b_{v}({\gamma_w}(t)),
{\gamma_w}'(t)\right> = 
\frac{d}{dt} ( b_{v}({\gamma_w}(t)) = -1 \;\;\mbox{by
(\ref{e10})}.$$
Since $\nabla b_{v}$ and ${\gamma_w}'(t)$ both are unit 
tangent vectors, we have
$$ \nabla b_{v}({\gamma_w}(t)) = - {\gamma_w}'(t).$$   
In particular,
$$  \nabla b_{v}(p)  = - {\gamma_w}'(0).$$
This implies that there is only one asymptotic geodesic
to $\gamma_v$
starting from $p$, namely in the
direction of  $\;-\nabla b_{v}(p).$
\end{proof}

\begin{Cor}\label{b2}
$(M,g)$ be an AH manifold.
Being asymptotic is an equivalence relation on $SM$ as
well as on the space of all oriented geodesics of $M$.
Moreover, two oriented geodesics 
are asymptotic if and only if the corresponding 
 Busemann functions agree upto an additive constant.
 \label{c9}
\end{Cor}
\begin{proof}                                     
 If $\gamma_w$
is the asymptotic geodesic asymptotic to ${\gamma_v}$
starting from $p$,  in  a non-compact AH  then, clearly $ b_{v}-b_{w}$  
is a harmonic function. By (\ref{e8}), it attains maximum
at $p$, and by maximum principle, it is a constant function. 
This implies that
$ b_{v}(x)-{b}_{w}(x)= b_{v}(p).$\\
\noindent
Conversely, suppose that 
 $ b_{v}(x)-{b}_{w}(x)= b_{v}(p).$
Therefore, 
\vspace{-0.05in}
\begin{eqnarray*}
\nabla b_v(\gamma_w(t)) & = & \nabla b_w(\gamma_w(t))\\
&=& -{\gamma_w}'(t).
\end{eqnarray*}
Thus, $\gamma_w$ is an integral curve of  $-\nabla b_v$.
But, from the Corollary \ref{c8} asymptotes are the integral 
curves of $-\nabla b_v$. Hence, we conclude that $\gamma_w$
is asymptotic to $\gamma_v$.
Consequently, it follows that being
asymptotic is an equivalence relation on $SM$. 
\end{proof}

\begin{Lem}\label{lem:bus}
 If  $\gamma_v$ is a geodesic line  in  an  AHM $M$,  then  $b_{v}^{+} + b_{v}^{-}= 0$.
\end{Lem}
\begin{proof}
  As $(M,g)$ is an AHM,  $ \Delta b_{v}^{\pm}= h=0. $ Also,
$$ b{_v}^{+}(x) + b_{v}^{-}(x) =  \lim_{t\rightarrow\infty} d
(x,\gamma_v(t)) + d (x,\gamma_v(-t)) -2t. $$ Hence, $b{_v}^{+}(x)
+ b_{v}^{-}(x) \geq 0$ for all $x$, by triangle inequality, and \\
$(b{_v}^{+} + b_{v}^{-})(\gamma_v(t))= 0$,
since $\gamma_v$ is a line.\\
\noindent Thus,  the minimum principle shows that $b_{v}^{+} +
b_{v}^{-}=0$.  
\end{proof}

\begin{Cor}
The stable horospheres,  $(b_v^+)^{-1}(c)$,  and unstable  horospheres,  $(b_v^-)^{-1}(-c)$, of an  AHM coincide like flat spaces.
\end{Cor}

\begin{Cor}\label{b3}
Any two asymptotic geodesics
of  an  AHM,  $M$, are  bi-asymptotic.
\label{c13}
\end{Cor}
\begin{proof}
From  Lemma \ref{lem:bus},  it  follows  that
\begin{eqnarray}
 b_{v}^{+}+b_{v}^{-}=0.
 \label{e16}
 \end{eqnarray}
Let ${\gamma_w}^{\pm}$ be unique asymptote to $\gamma_v$
starting from $p$ in positive and negative direction
respectively with ${\gamma_w}^{\pm}(0)=p.$\\ 
Therefore,  from the above  equation,
$$ \left(\nabla b_{v}^{+})(\gamma_{w}^{+}(0)\right)=
- \left(\nabla b_{v}^{-})({\gamma_w}^{-}(0)\right).$$

It follows that         
$$\frac{d}{dt}(\gamma_{w}^{+}(t))|_{t=0}=
-\frac{d}{dt}({\gamma_w}^{-}(t))|_{t=0}.$$   
Thus, two asymptotes ${\gamma_w}^{\pm}$ fit together to
form a smooth geodesic $\gamma_w.$  Hence, $\gamma_w$
is bi-asymptotic to ${\gamma_v}$.
\end{proof}

\begin{Cor}\label{bi-asy}
 If  $c_1$ and $c_2$  are two geodesics of  AHM,  $M$
such that \\ $d(c_1(t),c_2(t)) \leq C_1$ as $t \rightarrow \infty$, then
 $c_1$ is bi-asymptotic  to $c_2$. We obtain, \\ $d(c_1(t),c_2(t)) \leq C_2$ as $t \rightarrow - \infty$,
 consequently  $d_{H}(c_1, c_2)$,  Hausdorff distance  between  two  lines  is  bounded.
 We conclude, in an AHM,  there  exists  a  bounded  strip in all  directions.
\end{Cor}
\begin{proof}
If $c_1'(0) = v$ and $c_2'(0) = w$,
then by hypothesis $b^+_v - b^+_w \leq C_1$. Clearly,
$b^+_v - b^+_w$ is a bounded harmonic function.
We will prove Liouville  property for AHM (Corollary  \ref{Lio2})
in $\S3$. Hence, $b^+_v - b^+_w$ is constant by  Liouville  property on an AHM  $M$. Therefore, $b^+_v - b^+_w = \mbox{constant}$
implies that $c_1$ and $c_2$ are asymptotic towards $+\infty$
(from  Corollary \ref{c9}).  Since 
$b_v^+ = - b_v^-$ from (\ref{e16}),   we  also  have  that
$b_v^- - b_w^-$  is  also constant,  and in turn  
$c_1$ and $c_2$ are asymptotic towards $-\infty$  also.
Consequently,  using 
(\ref{e9}),  $d(c_1(t),c_2(t)) \leq C_2$ as $t \rightarrow - \infty$. 
\end{proof}

\noindent
{\it  Ideal  Boundary:}   On $SM$
define an equivalence relation $v \sim w,$ if $v$ is asymptotic to $w$. The equivalence classes of this 
relation are called {\em points at infinity}.
If $M(\infty) = SM/\sim$, then  from  Corollary  \ref{b2}
and  \cite{RS.03},  we  conclude  that   $M(\infty)$
{\it consists  of  equivalent   classes  of  Busemann functions
on  $M$.}  And  $\bar  M  =  M \cup  M(\infty)$  is  Busemann 
compactification of  an  AH manifold.  
$M(\infty)$  is  called  as   Busemann   boundary  or  ideal 
boundary  of  $M$.

\begin{Thm}\label{norm} 
Let $(M,g)$ be an AH manifold. Let $J_1, J_2, \cdots, J_{n-1}$ be the Jacobi fields
along a geodesic ray $\gamma_{v}$, arising from variations of asymptotic geodesics.
If $\{e_i\}$ is an orthonormal basis of $v^{\perp}$ and $J_{i}(0) = e_{i}$, then 
\begin{eqnarray}
||J_1 \wedge J_2 \cdots \wedge J_{n-1}||(t) = e^{-ht}.
\end{eqnarray}
\end{Thm}
\begin{proof}
Let $\phi_t$ be the flow of $\nabla b_{v}$ i.e., it
is  a one parameter family of diffeomorphism of $M$. 
Since, the  integral curves of $\nabla b_{v}$  are asymptotic
geodesics by Corollary \ref{b1}, $\phi_t$ is a variation of $\gamma_v$
through asymptotic geodesics. Therefore, $J_i(t) = d\phi_t(e_i)$. 
If $\omega$ denotes the volume form on $M$, then
$$ |\!|J_1 \wedge J_2 \wedge \cdot\cdot\cdot
\wedge J_{d-1}|\!|(t)
= |\det\;d\phi_t| =
|d\phi_t(e_1) \wedge d\phi_t(e_2) \wedge \cdot\cdot\cdot
\wedge d\phi_t(e_{d-1})| = |\phi_t^*\omega|.$$
Using properties of Lie derivative we get,
$$ - \frac{d}{dt} \phi_t^{*}\;\omega = L_{\nabla b_{v}}\omega =
({\rm div}\nabla b_{v})\; \phi_t^{*}\;\omega =
(\Delta b_{v}) \; \phi_t^{*}\;\omega.$$
This implies that
$ - \frac{d}{dt} \phi_t^{*}\;\omega =
h  \phi_t^{*}\;\omega.$ But, as
$ \phi_0^{*}\;\omega = \omega,$
we obtain that $\phi_t^{*}\;\omega = e^{-ht} \omega$.
This gives
$$|\!|J_1 \wedge J_2 \wedge \cdot\cdot\cdot
 \wedge J_{n-1}|\!|(t) = 
 |\det\; d\phi_t| = e^{-ht}.$$
\end{proof}

\begin{Cor}
If $(M,g)$ is an AHM, then there exists a bounded strip in all  directions.
\end{Cor}
\begin{proof}
By Theorem \ref{norm} for an AHM,   
$ ||J_1 \wedge J_2 \wedge \cdot\cdot \cdot\wedge J_{n-1}|\!|(t) = |{\det}\; d\phi_t| = 1.$ Hence, in an AHM there  exists at least one Jacobi field $J$ along a
geodesic $\gamma$ which is bounded. If $\tilde \gamma$ is a geodesic in the  
corresponding geodesic variation of $\gamma,$ then 
 $d(\gamma(t),\tilde \gamma (t)) \leq
 \displaystyle\int_{0}^{l}{|\!|J(s)|\!|}\;\mbox{ds} \leq C_1l.$
We conclude that $d(\gamma(t),\tilde \gamma (t)) \leq C$ for all $t > 0$.
By Corollary \ref{bi-asy}, $\gamma, \tilde \gamma$ are bi-asymptotic
geodesics and $d(\gamma(t),\tilde \gamma (t))$ is bounded for all $t \in \R$.
Hence, there exists a bounded strip in all  directions. 
\end{proof}

\noindent
\begin{Rem}
Let $J$ be a Jacobi field 
arising from a variation  of asymptotic geodesics.
If $J(r) = 0$ for some $r$, then $J'(r) = u^{-}(J(r))= 0$.  
Hence, $J \equiv 0$. Therefore, any Jacobi field arising from the
variation of asymptotic geodesics is non vanishing. 
\end{Rem}

\begin{Cor}
Let $(M,g)$ be an AH manifold  and let $\phi_{t}$ denote the geodesic 
flow of $M$. Let $H_{0} = b_{v}^{-1}(0)$, be the horosphere of $M$.  
If $D$ is any domain in $H_0$ and if $\phi_t(D) = D_t,$ 
then $A(D) = e^{ht}A(D_t).$ In particular, for AHM's $A(D) = A(D_t)$ and for an AHM 
$\phi_t$ is an area preserving diffeomorphism of $M$.
\end{Cor}
\begin{proof}
Since $\phi_t(D) = D_t$, by Theorem \ref{norm}, we obtain:
$$ A(D_t) = \int_{D_t} \omega = \int_{D} |\det d\phi_t| \omega
= \int_{D}  |\phi_t^* \omega|
= e^{-ht} A(D).$$
\end{proof}

Proposition $5.1$ of \cite{RS.02b} shows that HM's 
satisfy the absolute area minimizing property of horospheres. 
The proof also holds for AHM's and we have that AHM's satisfy the absolute area minimizing property of horospheres. 

\noindent
\begin{Pro}\label{am} Let $M$ be an AHM and
$D$ be any compact subdomain  of $H_0$
with $\partial D = \Gamma \subset H_0$. Let $\Sigma$ be any
compact hypersurface  of $M$ with
$\partial \Sigma = \Gamma \subset H_0.$
Then, $A(D) \leq A \Sigma$.
\end{Pro}

\section{Strong Liouville Type  Property}
Let $(M,g)$ be a complete, connected non-compact Riemannian manifold of infinite injectivity radius and of subexponential volume growth.     
In  this  section, we show that  $(M,g)$ satisfies  {\it Strong  Liouville Type Property} and hence {\it Liouville  Property}. We also show that it also 
satisfies {\it $L^1$ Liouville  Property}.\\\\
As AHM's have subexponential volume growth, we infer that they satisfy all
the aforementioned  Liouville Type Properties in comparison  with an HM  \cite{RS.02b}  (which  are now  known  to be flat spaces). Strong  Liouville Type Property of AHM is  an essential tool in proving the existence  of  Killing  vector field  on  AHM,  a  major  step for  proving  flatness of  AHM.\\\\
To prove the main result of this section, we use techniques of \cite{T.08} and \cite{RS.02b}. We recall that in \cite{RS.02b} the integral formula for the derivative of harmonic function on harmonic manifolds was proved.\\\\
Let $(M,g)$ denote a complete, connected and non-compact Riemannian manifold of infinite injectivity radius. In this section, we obtain an integral formula for the derivative of mean value of a $C^1$ function on $(M,g)$, viz., Theorem \ref{IDM}, of this section.  In particular, Theorem \ref{IDM} also holds for AH manifolds.
In fact, Theorem \ref{IDM} is the strengthening of the main result of 
\cite{T.08} viz., Theorem \ref{min1}, as well as strengthening of the integral 
formula of \cite{RS.02b}  viz., Theorem \ref{HL}, quoted in subsection $3.1$.

\begin{Def}\label{sl} $(M,g)$  be  a  non-compact Riemannian manifold
\begin{itemize}
\item[1)] $(M,g)$ is said  to  satisfy  {\it Strong Liouville Type  Property} if there are  no  non-trivial  subharmonic  functions  on  $(M,g)$ which  are bounded from above.
\item[2)] $(M,g)$  is  said  to  satisfy {\it Liouville  Property} if there are  no  non-trivial bounded harmonic functions on $(M,g)$.
\item[3)] $(M,g)$ is said to satisfy {\it $L^1$-Liouville  Property} if there are no non-trivial non-positive subharmonic functions in $L^1(M)$.
\end{itemize}
\end{Def}

\subsection{Known integral formulae for the derivative of harmonic and subharmonic functions}
In this subsection we recall known integral formulae for the derivative of harmonic and subharmonic functions. We also recall definitions and results needed to prove Strong Liouville Type Property for manifolds of subexponential volume growth.\\\\
{\it Definition : Volume growth type}
\begin{itemize}
\item[(i)] If  volume  growth  of  a  manifold  $M$
 satisfies  that $\Vol(B(p,r)) \geq c^r $ for large $r$,
 and  for  some  $c > 1$,  then  it  is  said  to have  
 {\em exponential volume growth.} 
\item[(ii)] If  volume  growth  of  a  manifold  $M$
 is not exponential, then $M$ is said to have {\em subexponential volume growth}.
Equivalently, volume growth of a manifold is subexponential, if 
$\displaystyle \lim_{r \rightarrow \infty}\frac{\log V(p,r)}{r} = 0$.  
\item[(iii)] If  $\Vol(B(p,r)) \leq C r^{n}\;$ for large $r$ and for some
$C$ and $n > 0,\;$ then $M$ is said to have
{\em polynomial volume growth}.
\end{itemize}
\noindent 
Note  that from the above Definition (ii) an  AHM  has  apriori subexponential volume  growth.\\  
In the sequel, $(M,g)$ tacitly denotes a complete, non-compact Riemannian manifold of infinite injectivity radius. Let $V(p,r)$, $A(p,r)$ and $S$, respectively, denote the volume of the ball $B(p,r)$, area of the sphere $S(p,r)$ and unit sphere
in $T_{p}M$, respectively.

\begin{Def}
Let $u$ be  a continuous (or measurable) function on $M$. 
Fix $p \in M$ and $r >0$. Then {\it the mean value of $u$ at $p$} is its average over the ball centred at $p$ and radius $r > 0$, denoted by $A_{u,r}(p)$. We have:
\begin{equation}\label{mean}
A_{u,r}(p) = \frac{1}{V(p,r)} \int_{B(p,r)}u \;d\mu.  
\end{equation}
\end{Def}

\noindent
\begin{Def} \label{WMV}   
A manifold $(M,g)$ is said to satisfy {\it weak mean value inequality}, 
${\cW \cM}_R(\lambda, b)$, if there exists a constant $\lambda > 0, b >1$
such that for any $r \leq R / b $ and $f(x) \geq 0$ satisfying 
$\Delta f \geq 0$ on $\;B_p(rb),$ then
$$ f(p)  \leq \frac{\lambda}{\Vol(B_p(r))} \int_{B_p(br)} f(x) dx.$$
\end{Def}
\noindent
Note :  If  $b =1$ in  ${\cW \cM}_R(\lambda, b)$,  then  it  is  called  as
{\it Mean Value Inequality ${\cM}_R(\lambda)$. \label{MV}}  
\noindent
It is well known that all
harmonic manifolds satisfy Mean Value Inequality ${\cM}_R(\lambda)$ with 
$\lambda = 1$ for all $R$. See  \cite{W.50}.\\

We will use the following {\it stronger version of the maximum principle} for the subharmonic  functions. Lemma 3.4 of \cite{GT.83} states that:
\begin{Thm}\label{smax}
At the point on the boundary where the maximum is attained, the subharmonic function
has a positive outward normal derivative. 
\end{Thm}

\noindent
Now we recall the definition of {\it the stability vector field} as defined in \cite{T.08}.
\begin{Def}
For any real number $r > 0$, the {\it stability vector field} $H(.,r)$ on $(M,g)$ is defined by
\begin{eqnarray} 
H(p,r) = \int_{B(p,r)} exp_{p}^{-1}(q)\; d \mu(q),\;\; \forall p \in M,
\end{eqnarray}
where $exp_{p}^{-1}$ denotes the inverse of exponential map.
\end{Def}
\noindent
It was proved in \cite{T.08}
that the volume function $V$ and the stability vector field $H$ are related by the following differential equation.
\begin{Lem}\label{derv}
Let $\nabla$ denote the gradient operator on $(M,g)$. Then
for any $r > 0$ and $p \in M$, it holds:
\begin{eqnarray*}
\nabla V(p,r) - \frac{1}{r} \frac{\partial}{\partial r} H(p,r) = 0. 
\end{eqnarray*}
\end{Lem}
We recall definition of  angle as defined in \cite{RS.02b}.  
\begin{Def}
Let $p \in M$ and $x \in S_{p}M$. Define {\it angle} as in \cite{RS.02b}, the function:
\begin{eqnarray*}
\theta_{x} : M  \setminus {p} &\rightarrow& \R \\
  q &\rightarrow & \theta_x(q) = \angle_p(x, v),
\end{eqnarray*}
where $\gamma_{v}$ is the unique geodesic joining $p$ to $q$.
\end{Def}

In \cite{RS.02b} the integral formula for the derivative of harmonic functions was proved which led to the proof of Liouville Theorem for HM.

\begin{Thm}\label{HL}
Let $u : M \rightarrow \R$ be a harmonic function on a harmonic manifold $M$.
Then for $x \in S_{p}M$,
\begin{eqnarray}\label{harder}
xu(p) = \frac{1}{V(p,r)}\int_{S(p,r)} u \cos{\theta_{x}} \; d \sigma, 
\end{eqnarray} 
where $d \sigma$ denotes the volume of $S(p,r)$. 
\end{Thm}
The above integral formula was generalized in \cite{KP.13} to obtain the 
integral formula for the derivative of subharmonic functions on a harmonic manifold $M$.
\begin{Thm}\label{SHL}
Let $u : M \rightarrow \R$ be a subharmonic function on a harmonic manifold $M$.
Then for $x \in S_{p}M$,
\begin{eqnarray}\label{subharder}
\left<\nabla u(p), x\right> \leq \frac{1}{V(p,r)}\int_{S(p,r)} u \cos{\theta_{x}} \; d \sigma.
\end{eqnarray} 
\end{Thm}

The integral formula in Theorem \ref{HL} above was generalized further in \cite{T.08}.

\begin{Thm}\label{min1}
Let $(M,g)$ be a non-compact, connected, complete Riemannian manifold with infinite injectivity radius. If $u$ is a function on $M$ satisfying the mean-value property, then
it holds:\\
For any real number $r > 0$, 
\begin{eqnarray}\label{der}
xu(p) = \frac{1}{V(p,r)}\int_{S(p,r)} u \cos{\theta_{x}} \; d \sigma 
- \frac{1}{r}\frac{u(p)}{V(p,r)} \left<\nabla V(p,r), x \right>.  
\end{eqnarray}   
\end{Thm}

\begin{Cor}\label{minhor}
Let $(M,g)$ be a non-compact, connected  and complete Riemannian manifold with minimal
horospheres and infinite injectivity radius. Any bounded function on $(M,g)$
satisfying mean value property is constant. 
\end{Cor}

\subsection{Integral formula for the derivative of mean value of $C^1$ function}
In this subsection, we prove our main result viz., the Strong Liouville Type Property,
for a complete, connected, non-compact Riemannian manifold of infinite injectivity radius and of subexponential volume growth. In particular, we obtain the desired result for AHM. \\\\
First we prove the integral formula for the derivative of mean value of a $C^1$  function on $(M,g)$, a complete, connected, non-compact Riemannian manifold of infinite injectivity radius. The proof uses techniques of the proof of Theorem \ref{min1}
of \cite{T.08} and proof of Theorem \ref{HL} of \cite{RS.02b}.

\begin{Thm}\label{IDM}
Let $(M,g)$ be a complete, connected, non-compact Riemannian manifold of infinite injectivity radius. Let $u$ be a $C^1$ differentiable function on $M$. Then the derivative of mean value of $u$ satisfies the integral formula: 
\begin{eqnarray}\label{mu}
x.A_{u,r}(p) & = & -\frac{A_{u,r}(p)}{V(p,r)} \left<\nabla V(p,r), x \right>  \\ \nonumber
& + & \frac{1}{V(p,r)} \int_{S(p,r)} u \cos \theta_{x} \;d \sigma.
\end{eqnarray}
\end{Thm}
\begin{proof}
Let $u$ be a $C^1$ function on $M$. 
Fix $p \in M$ and $r >0$. We have:
\begin{equation}\label{mean}
A_{u,r}(p) = \frac{1}{V(p,r)} \int_{B(p,r)}u \; d\mu.  
\end{equation}
Let $c$ be the geodesic with $c(0) = p$  and $c'(0) = x$.
Then, 
\begin{eqnarray*}
x.A_{u,r}(p) & = & \frac{d}{dt} A_{u,r}(c(t))|_{t=0} \\ \nonumber
&=&\frac{d}{dt}\left(\frac{1}{V(c(t),r)} \int_{{B(c(t),r)}} u d \mu \right)_{|t = 0}.
\end{eqnarray*}
Consequently,
\begin{eqnarray}\label{md}
x.A_{u,r}(p)
&=&-\frac{1}{V(p,r)^2} <\nabla V(p,r), x> \int_{B(p,r)} u \; d \mu  \\ \nonumber
& + &\frac{1}{V(p,r)} \frac{d}{dt} \left(\int_{{B(c(t),r)}} u \; d \mu \right)_{|t = 0} 
\end{eqnarray}
From (\ref{mean}) we obtain
\begin{eqnarray}\label{grad} 
\frac{1}{V(p,r)^2} <\nabla V(p,r), x> \int_{B(p,r)} u \; d \mu & = & \frac{A_{u,r}(p)}{V(p,r)} \left<\nabla V(p,r), x \right>. 
\end{eqnarray}
Using techniques of proof of Theorem 2.1 of \cite{RS.02b} we will show that
\begin{eqnarray}\label{angle}
\frac{d}{dt} \left(\int_{B(c(t),r)} u d \mu \right)_{|t = 0} =  \int_{S(p,r)} u \cos \theta_{x} \;d \sigma.
\end{eqnarray}
Finally, substituting (\ref{grad}) and (\ref{angle}) in (\ref{md}) we obtain the 
required integral formula (\ref{mu}). \\\\
Now it remains to prove (\ref{angle}).\\\\
Consider one parameter family of diffeomorphisms of M,
as in \cite{RS.02b},
given by 
\begin{eqnarray*}
f_{t} = exp_{c(t)}\circ P_{t}\circ exp_{p}^{-1} : M 
\rightarrow M. 
\end{eqnarray*}  
Let $J_x = J$ be the vector field induced by $f_t$ at time $t = 0$. Since $f_t$ maps radial geodesic starting from $p$ to geodesics starting from $c(t)$, J satisfies the following property which we shall make use of:\\
$J$ is the unique Jacobi field satisfying $J(0) = x, J'(0)= 0$, when restricted to any geodesic starting from $p$.\\\\
As $f_t(B(p,r)) = B(c(t),r),$ we obtain  
\begin{eqnarray*}
\left(\int_{B(c(t),r)} u d \mu \right) & = & \left(\int_{B(p,r)} {f_t}^*(u d\mu)\right)  =  \left(\int_{B(p,r)} {f_t}^*(u) {f_t}^*(d\mu)\right) 
\end{eqnarray*}
We have ${f_t}^*(d\mu) = |\det df_{t}|\; d\mu.$ 
But as $J$ satisfies $J(0) = x, J'(0)= 0$, when restricted to any geodesic starting from $p$; the Jacobian of ${f_t}$ at $t = 0$ viz., $|\det {df_{t}}_{|t = 0}| = 1$,
and $\displaystyle\frac{d}{dt} \det {df_{t}}_{|_{t = 0}} = 0.$
Therefore,
\begin{eqnarray*}
\frac{d}{dt} \left(\int_{B(c(t),r)} u d \mu \right)_{|t = 0} & = &  
\frac{d}{dt}\left(\int_{B(p,r)}({f_t}^*(u\;d\mu))  \right)_{|t = 0}\\
& = & \int_{B(p,r)} \frac{d}{dt}({f_t}^*(u))_{|t = 0} \; d\mu
\end{eqnarray*}
Consequently,
\begin{eqnarray*}
\frac{d}{dt} \left(\int_{B(c(t),r)} u d \mu \right)_{|t = 0} & = & 
 \int_{B(p,r)} L_{J}u \; d \mu =  \int_{B(p,r)} L_{J}{(ud \mu)} \\
 & = & \int_{S(p,r)} i_{J}{(ud \mu)}, 
\end{eqnarray*}
where $i_{J}$ denotes the interior product with the field $J$.
However, $\omega = dr \wedge d\sigma$ and $dr|_{S(p,r)} = 0$. 
Therefore, $i_{J}(u\omega) = u i_{J}(\omega)$. Also 
$i_{J}dr = <J, \frac{\partial}{\partial r}>$.  Hence, 
$i_{J}(u\omega) =  u <J, \frac{\partial}{\partial r}> d\sigma$ and we obtain,
\begin{eqnarray*}
\frac{d}{dt} \left(\int_{B(c(t),r)} u d \mu \right)_{|t = 0} & = & 
\int_{S(p,r)}  u <J, \frac{\partial}{\partial r}> d\sigma.
\end{eqnarray*}
Since $J$ is a Jacobi field with $J(0) = x, J'(0) = 0$, its parallel component is
\begin{eqnarray*}
<J, \frac{\partial}{\partial r}> & = & <J(0), v> + <v, J'(0)>r
 =  \cos \theta_{x}.
\end{eqnarray*} 
Finally we obtain,
\begin{eqnarray*}
\frac{d}{dt} \left(\int_{B(c(t),r)} u d \mu \right)_{|t = 0} =  \int_{S(p,r)} u \cos \theta_{x} \;d \sigma.
\end{eqnarray*}
\end{proof}

\begin{Cor}
If $u$ is a harmonic function on a harmonic manifold $(M,g)$, then we recover
Theorem \ref{HL}. 
\end{Cor}
\begin{proof}
By a characterization of harmonic manifolds, any harmonic function on a harmonic manifold satisfies mean value property. And the volume density function $\Theta_{p}$
is independent of point $p \in M$. Therefore, in this case $\nabla V(p,r) = 0$
and $A_{u,r}(p) = u(p)$. Therefore, the conclusion follows from the above
Theorem \ref{IDM}. 
\end{proof}

Now we can prove our main theorem of this section.

\begin{Thm}\label{SLTP}
If $(M,g)$ is a complete, connected, non-compact Riemannian manifold of infinite injectivity radius and of subexponential volume growth, then there are no non-trivial bounded subharmonic functions on $(M,g)$. In particular, an  AHM also has no non-trivial bounded subharmonic functions.           
\end{Thm}
\begin{proof}
Let $u$ be a bounded subharmonic function on $M$. Then there exists a constant
$\alpha > 0$ such that $|u| \leq \alpha$. From the above integral formula, $\forall p \in M \; \mbox{and} \; r > 0$ and Lemma \ref{derv} we obtain,
\begin{eqnarray}\label{u}
||x.A_{u,r}(p)|| \leq \frac{\alpha}{V(p,r)}||\frac{1}{r}\frac{\partial}{\partial r} H(p,r)|| + \alpha \frac{A(p,r)}{V(p,r)}. 
\end{eqnarray} 
 But:\\
\begin{eqnarray*}
||\frac{1}{r}\frac{\partial}{\partial r} H(p,r)|| & = & ||\frac{1}{r}\frac{\partial}{\partial r} \int_{B(p,r)} exp_{p}^{-1}(q) d \mu(q)||\\
& = & ||\frac{1}{r}\frac{\partial}{\partial r} \int_{S(p,r)} exp_{p}^{-1}(q) d \sigma(q)||\\
& \leq & \frac{1}{r} \int_{S(p,r)} ||exp_{p}^{-1}(q)|| d \sigma(q)\\
& = & A(p,r),\; \mbox{since}\; ||exp_{p}^{-1}(q)|| = r,\; \forall q \in S(p,r).
\end{eqnarray*}
Consequently, 
\begin{eqnarray}\label{u}
||x.A_{u,r}(p)|| \leq 2 \alpha \frac{A(p,r)}{V(p,r)}.
\end{eqnarray}
We have,
\begin{eqnarray}\label{sub}
\lim_{r \rightarrow \infty} \frac{A(p,r)}{V(p,r)} \leq  \lim_{r \rightarrow \infty} \frac{\int_{S}\Theta_{p}(r,u) d\sigma}{\int_{0}^{r} \int_{S}\Theta_{p}(r,u) dr d\sigma}  = \lim_{r \rightarrow \infty} \frac{\Theta_{p}'(r,u)}{\Theta_{p}(r,u)}=0.
\end{eqnarray} 
Hence, for any $p \in M$ and for any $x \in S_{p}M$,
\begin{eqnarray}\label{mader}
||x.A_{u,r}(p)|| = 0 \; \mbox{as}\; r \rightarrow \infty.
\end{eqnarray}
We have, 
\begin{eqnarray}\label{mdr}
V(p,r) A_{u,r}(p) = \int_{B(p,r)}u\;d\mu =  \int_{0}^{r} \int_{S} u(r,\phi)\Theta_{p}(r,\phi)\;dr d\phi.  
\end{eqnarray}
Therefore, differentiating (\ref{mdr}) with respect to $r$, 
\begin{eqnarray}
V'(p,r)A_{u,r}(p) + A_{u,r}'(p)V(p,r) = 
\int_{S} u(r,\phi)\Theta_{p}(r,\phi)\;dr d\phi.   
\end{eqnarray}
Therefore,
\begin{eqnarray}
 A_{u,r}'(p) =
\frac{1}{V(p,r)} \int_{S} u(r,\phi)\Theta_{p}(r,\phi)\;dr d\phi - \frac{V'(p,r)}{V(p,r)}A_{u,r}(p)   
\end{eqnarray}
As
$$V(p,r) = \displaystyle \int_{0}^{r} \int_{S} \Theta_{p}(r,\phi)\;dr d\phi,$$
we have
\noindent
$V'(p,r) = \displaystyle \int_{S} \Theta_{p}(r,\phi)\;dr d\phi = A(S(p,r)).$
Consequently,
\begin{eqnarray}
|A_{u,r}'(p)| \leq 2 \alpha \frac{A(p,r)}{V(p,r)}. 
\end{eqnarray}
From (\ref{sub}) it follows that 
\begin{eqnarray}\label{rdr}
\lim_{r \rightarrow \infty} |A_{u,r}'(p)| = 0.
\end{eqnarray}
Thus  from  (\ref{mader}) and (\ref{rdr}), $A_{u,r}(p)$ is independent of 
$p \in M$ and also independent of $r$ for sufficiently large $r$. Hence, this clearly implies that $u$ is  constant on all balls of sufficiently large radii.\\
We obtain that $u(x)= C = \mbox{constant},\;\; \forall {x \in B(p,r), r \geq N,}$
for $N$ sufficiently large.   
By the maximum principle of subharmonic functions,
$$\max_{x \in B(p,r)} u(x) = \max_{y \in \partial B(p,r)} u(y) = u(z),\; 
\mbox{for some $z$ with r=d(p,z)}.$$ 
Now by the stronger version of the maximum principle for the subharmonic 
functions, Theorem \ref{smax}, we obtain :
$$\displaystyle \max_{x \in B(p,r), r < N} u(x) = \max_{x \in \partial B(p,r), r < N}  u(x) \leq \max_{x \in \partial B(p,r), r \geq N} u(x) = C.$$ 
Therefore, for all $x \in M$, we have $u(x) \leq C$ i.e. $u$ attains global 
maximum on $M$.  Therefore, $u$ is  a constant function on $M$, again by the maximum 
principle for subharmonic functions.    
\end{proof}

We observe that Liouville Property holds on an AHM, in  comparison  with  
an HM.  

\begin{Cor} \label{Lio1} 
$(M,g)$ be a complete, connected, non-compact Riemannian manifold of infinite injectivity radius and of subexponential volume growth. Then $(M,g)$ satisfies Liouville Property on $M$. In particular, an  AHM also satisfies Liouville Property.   
\end{Cor}

\begin{Cor}\label{Lio2} 
$(M,g)$ be a complete, connected, non-compact Riemannian manifold of infinite injectivity radius and of subexponential volume growth. Then $(M,g)$ satisfies Strong Liouville Type Property. In particular, an  AHM also satisfies Strong Liouville Type Property.
\end{Cor}
\begin{proof}
If $u$  is  a  subharmonic function on $M$ bounded from  above, then
$g(x)  =  e^{u(x)}$  is a bounded  non-negative  subharmonic  function on  $M$,  as
$\Delta  g  =   e^{u(x)}  (\Delta  u  +   || \nabla u||^2)   \geq  0$. Therefore,  by  Theorem \ref{SLTP},  $g$  is  a  constant  function and in turn,  $u$ is  a  constant  function.
\end{proof}

\begin{Cor}\label{bdder} 
$(M,g)$ be a complete, connected, non-compact Riemannian manifold of infinite injectivity radius and of exponential volume growth. 
If $u$ is a bounded function on $M$ satisfying mean value property, then the derivative of $u$ is also bounded. In particular, harmonic functions on harmonic manifolds
of exponential volume growth satisfy this property.  
\end{Cor}
\begin{proof}
As $(M,g)$ is of exponential volume growth, we have 
\begin{eqnarray*}
\lim_{r \rightarrow \infty} \frac{A(p,r)}{V(p,r)} \leq  \lim_{r \rightarrow \infty} \frac{\int_{S}\Theta_{p}(r,u) d\sigma}{\int_{0}^{r} \int_{S}\Theta_{p}(r,u) dr d\sigma}  = \lim_{r \rightarrow \infty} \frac{\Theta_{p}'(r,u)}{\Theta_{p}(r,u)}=h >0.
\end{eqnarray*}
Therefore, from the proof of Theorem \ref{SLTP}, $||xA_{u,p}(r)|| = |xu(p)| \leq 2 \alpha h,$ for all $p \in M$.
\end{proof}

\noindent
Now  we  compare  results  of  this  section  with  some similar
type of  results  on   Liouville   Property.\\

\noindent
Li and  Wang \cite{LW.99} have  proved the Liouville property in
the set up of the theorem given below. 
\begin{Thm}\label{m} $(M,g)$ be a complete manifold satisfying the
mean value inequality
$\cM(\lambda)$ i.e. ${\cM}_R(\lambda), \; \forall R$.
If $\lambda < 2$ and $M$ has
subexponential volume growth, then
bounded harmonic functions are constants.
\end{Thm}

In fact, a slight modification of their proof of Theorem \ref{m} in \cite{LW.99} shows:

\begin{Thm}\label{m1} $(M,g)$ be a complete manifold of
subexponential volume growth, then $(M,g)$ satisfies Strong Liouville Type Property.
\end{Thm}  

\begin{Thm}\label{l1}
Let $(M,g)$ be a manifold of subexponential volume growth, then 
any non-negative superharmonic function $u \in L^1(M)$ is  constant. 
\end{Thm}
\begin{proof}
The proof follows from Theorem 13.2 of \cite{G.96}.
\end{proof}

\begin{itemize}
\item[1)] Note that for manifolds of subexponential volume growth, 
we obtain  Liouville  property as an application of Theorem \ref{m1}.
{\it On the other hand Liouville  property, Corollary  \ref{Lio1}, 
obtained here by direct method and  also  without using the 
mean value property, as opposed to the usual practice}.
\item[2)] Note that the Strong Liouville Type Property for harmonic manifolds 
follows from Theorem \ref{SHL} directly. But, the proof of Theorem \ref{SHL}
 \cite{KP.13} uses harmonicity heavily. On the other hand Theorem 
\ref{IDM} is quite general and suffices to prove Strong Liouville Property
on manifolds of subexponential volume growth.
\item[3)] Note that conclusion of Corollary \ref{bdder} need not imply
that $u$ is a constant function. In fact, Strong Liouville Property {\it does not
hold on AH manifolds with constant $h >0$.} \cite{RS.02b} shows that on harmonic manifolds of $h >0$, for each $v \in S_{p}M$, 
$$h_{v}(x) = - n \displaystyle\frac{\int_{0}^{r} \Theta(r)}{\Theta(r)} \cos \theta_{v}(x),$$ is a family of non-constant bounded harmonic functions. \\
\end{itemize}

\section{Existence of Killing Vector  Fields on AHM}

In  this  section,  we  show  that   any  Killing  vector  field   on  an  AHM,  $M$,  
is non-trivial and parallel. And thus, there exists a Killing  vector  field   on  an  AHM. Further, we  also  show  that  on  an   AH  manifold  of  $h >0$,
there does  not  exist  any  Killing  vector  field  of  constant 
length.\\

First  we  describe  some  preliminaries  on  Killing field.

\subsection{Killing  Vector  Fields}

\begin{Def}
A vector field $X$ on a Riemannian manifold $(M,g)$ is called a {\it Killing field}, if the local flows generated by $X$ acts by isometries. 
 Equivalently, $X$ is  a  Killing field if  and  only  if   $L_X g = 0$,  where $L$  denotes the Lie  derivative of  metric $g$  with respect to  $X$.  
 \end{Def}

Now  we  recall  the  following  results  about   Killing  vector fields.   Proofs  of  the results  can be  found in  \cite{P.98}.

\begin{Pro}\label{k1}
$X$  is  Killing  field  if  and  only  if  $v \rightarrow  \nabla_{v}X$ is a skew symmetric $(1,1)$-tensor.  
\end{Pro}

\begin{Pro}\label{k2}
For  a  given $p  \in  M$,  a  Killing  vector  field  $X$  is  uniquely  determined  by  $X(p)$  and 
$(\nabla X)(p)$.
\end{Pro}

\begin{Pro}\label{dim}
The set of Killing fields $iso(M,g)$ is a Lie algebra of dimension $\leq \frac{n(n+1)}{2}$. 
Furthermore, if $M$ is complete, then iso$(M,g)$ is the Lie  algebra of Iso$(M, g)$.
\end{Pro}

\begin{Pro}\label{delta}
If $X$  is a Killing field on $(M,g)$ and consider the function  $f  =  \frac{1}{2} g(X,X)  =  \frac{1}{2} ||X||^2$,
  then 
 \begin{itemize}
 \item[(1)]  $\nabla f  =  -  \nabla_{X} X$.
 \item[(2)] Hess $f (V, V ) = g (\nabla_{V} X,  \nabla_{V} X)  -  R (V, X, X, V )$.
 \item[(3)] $\Delta f =  ||\nabla  X||^2  -  Ric(X,X)$.
 \end{itemize}
 \end{Pro}

\begin{Lem}\label{hess-k}
If  $K$  is  a  Killing  field  on  a  Riemannian manifold $(M,g)$,  then 
\begin{eqnarray}\label{k1}
{\nabla^2_{X,Y}} K  =  - R(K,X)Y
\end{eqnarray}
\end{Lem}

We   refer  \cite{BN.08}  for more  geometric  exposition
on  Killing vector  fields. The  following  Proposition  \ref{c} can 
be  found  in   \cite{BN.08}. 
\begin{Pro}\label{c}
A  Killing  vector  field  $X$  on a Riemannian manifold $(M,g)$  has constant  length
if  and  only  if  every  integral  curve  of  the  field  $X$  is geodesic.  
\end{Pro}

\begin{Pro}\label{J}
If  $X$  is  a Killing  vector  field  on a Riemannian manifold $(M,g)$,
then  $X$  is  a  Jacobi  field  along  every  geodesic  $\gamma(t)$,
for  all  $t \in  \R$.
\end{Pro}

\subsection{Killing  Vector Fields  on AHM}   

\vspace{0.5cm}

 In  this  subsection,  we  prove the main
result of  this paper that, in  an  AHM every Killing  vector field  on  $M$  is  parallel.

\begin{Thm}\label{const}
 If $X$  is  a  Killing  vector  field on  an AHM $M$, then  there exists  a
constant  $C >0$  such  that  $||X|| \leq  C,$  and consequently, $X$ is  parallel on $M$.
\end{Thm}
\begin{proof}
Let $X$  be  a non-trivial Killing  vector  field  on an AHM  $M$. 
 From  Proposition  \ref{non-pos},  in  an  AHM  $M$,  $Ricci_{M} \leq  0$. By  Proposition  \ref{delta},  it  follows  that
$f  =   \frac{1}{2} g(X,X)  =  \frac{1}{2} ||X||^2$   is  a  non-negative, non-zero subharmonic  function  on $M$.   \\
Case $(i)$  Suppose  that  $||X||  <  C$, for  some $C > 0$, then $f$ 
is   a  bounded   non-negative subharmonic  function  on $M$.   By  Theorem
\ref{SLTP},  $f$  is  a  constant function  on  $M$.  Hence,
 $$ ||\nabla X||^2  =  Ricci(X,X)  \leq 0.$$   Consequently,  $||\nabla X||^2 = 0$ and thus $X$ is a parallel  vector  field.  \\

\noindent
Case $(ii)$  
Fix  $p  \in  M$.  Suppose  that  there  exists $x \in  M$ such  that \\ $||X||(x)  \rightarrow  \infty$  as  $d(p,x)  \rightarrow  \infty$  on  $M$.  Consider 
$g(x)  =  e^{-f(x)}$. Then $g$  is a non-negative  bounded  function on  $M$   with  $0  \leq  g \leq 1$.  We  have,
\begin{eqnarray} \label{g}
\Delta  g (x) =   e^{-f(x)}  ( -\Delta  f  +   || \nabla f||^2).
\end{eqnarray}
\begin{eqnarray} 
\mbox{If}\;\;\;   h(x)  =   -\Delta  f  +   || \nabla f||^2, \; \; \mbox{then}   \;\;  \Delta  g (x) =   e^{-f(x)}  h(x).
\end{eqnarray}

Note  that   $g$  is  a smooth  non-negative  function,  
 which  converges to zero outside  a compact  set  $K$ containing 
$p$.  \\\\
 (a)  Suppose   that  $h  \equiv  0$,  then  $g$   is a  bounded  harmonic  function 
on  AHM.  Then  $g$  is  constant  function on $M$  by  
Corollary  \ref{Lio1},  Liouville  Property.  If  $g(q)  =  0$  for some  $q  \in  K$,  then  
$g  \equiv  0$.  This  implies  that  $f$  converges  to  $\infty$  
everywhere. This  is  a contradiction  to  smoothness  of  $f$.\\
If  $g(q)  =  c >0,$  for some  $q  \in  K$,  then $g \equiv  c,$ 
which  again contradicts  to smoothness  of  $g$,  as  $g$
 converges to zero outside  $K$.  Therefore,  this  case  can't  occur.  \\\\
(b)   Suppose  that  there  exists  a  point  $q  \in  K$   such  that $h(q) > 0$.   Then  by  continuity  we  can  find
a  neighbourhood  of  $q$ say  $N \subset K$  in  which  $h >0$.  Then  $g$  is  bounded  subharmonic  function  on 
$N$.  Again  by  Theorem \ref{SLTP},  $g$  is  constant  function  on  $N$.   Note  that   by  argument  of  case (a),
$g = c > 0$  on  $N$.    Therefore,  $f$  is  a  positive constant  on  $N$  and  consequently,  $X$  is a non-trivial parallel  vector  field  
on  $N$. Therefore,  we  can assume $||X||  = 1$   on  $N$.  But  from Proposition  \ref{k2},  a Killing  vector  field  on $M$
is  uniquely determined  by  $X(q) = v,  \nabla X(q) = 0$.  Thus,  $X$  is  parallel  vector  field  on all of $M$.   But,  this  is  
a  contradiction  to   $||X||(x)  \rightarrow  \infty$.  Thus,   this  case  also can't  occur.    \\\\
We  conclude that, $X$  must  be  a bounded  vector  field  and by  case (i)  $X$  is parallel.
\end{proof}

Now we show  that  there exists a non-trivial  parallel  vector field on an AHM.
We need the  following important Lemma.

\begin{Lem}\label{dis}
If $(M,g)$ is  an AHM,  then Iso$(M,g)$ is  diffeomorphic to Iso$({\R}^n,can)$.
Consequently, Iso$(M,g)$ is diffeomorphic to semidirect product of $O(n)$ and 
${\R}^n$. Hence, in particular, Iso$(M,g)$ can't be a discrete group.  
\end{Lem}
\begin{proof}
As $(M,g)$ is  an AHM, then by geometry of $M$, it is diffeomorphic to
${\R}^n$. Let  $\phi : M \rightarrow {\R}^n$ be diffeomorphism. 
Define
\begin{eqnarray*}
Iso(M,g) \rightarrow  Iso({\R}^n,can)\\
\psi \rightarrow \phi \circ \psi \circ {\phi}^{-1}
\end{eqnarray*}
Note that as $\psi$ is an isometry of $M$ and $\phi$ is a diffeomorphism,
$\phi \circ \psi \circ {\phi}^{-1} :  {\R}^n \rightarrow {\R}^n$ 
is  a distance preserving surjective map, hence is  an isometry, by Myer-Steenrod 
theorem.  Clearly, the above map defines diffeomorphism between the two isometry 
groups.    
\end{proof}

\begin{Cor}\label{non-trivial}
If  $X$  is  a Killing  vector  field on  an AHM,  then $X$  is  a
non-trivial  parallel  vector field on  $M$.
\end{Cor}
\begin{proof}
If  $X$  is  a Killing  vector  field on  an  AHM,  then $X$  is  a
parallel  vector field on  $M$  by  Theorem \ref{const}.
Thus, either a Killing  vector  field  on an AHM is 
identically zero or is non-trivial and parallel.  
Suppose that there doesn't exist any non-trival Killing vector field on $M$.
Then, as  Iso$(M,g)$ is a Lie group, its connected component of identity 
is trivial. Hence, Iso$(M,g)$ must be a discrete group.    
This  is  a contradiction, in view of Lemma \ref{dis}, proves  the
 required result.
 \end{proof}

\begin{Cor}\label{dim}
In  an AHM  $M$, we  have $1\leq \dim iso(M,g) \leq \dim M = n$.      
\end{Cor}
\begin{proof}   
Since  any  Killing field $X$ on $M$ is a  non-trivial parallel vector field,  the  linear  map, 
 $X  \rightarrow  X(p)$  is  injective from  $iso(M,g)$  to  $T_pM$. 
 And  therefore, $1\leq \dim iso(M,g) \leq  \dim M = n$.
\end{proof}

\begin{Cor}\label{Jac}
In  an  AHM $M$, there exists a non-trivial 
Jacobi  field on  $M$  which  is  parallel, and which arises out of variation
of  bi-asymptotic  geodesics.
\end{Cor}
\begin{proof}   
Note  that  from  Corollary \ref{non-trivial}, a  Killing field  $X$ on AHM is non-trivial  parallel  and  
by  Proposition  \ref{J},  $X$  is  Jacobi  field  along  every  geodesic. 
We  may  assume  that  $||X|| =1$.   Let  $X(p)  =  v$.  Consider  geodesic 
 $\gamma_{w}$,  passing  through $p$  with  initial  velocity  vector  $w  \perp v$.  
Then  $X$  is  parallel perpendicular  Jacobi  field  along  $\gamma_{w}$.
From  \cite{E.77},  it  follows  that  $X$  arises out  of  variation of 
bi-asymptotic  geodesics,  bi-asymptotic  to  $\gamma_{w}$. 
\end{proof}

Thus,  we  have shown:

\begin{Cor}\label{cl}
In  an  AHM  $M$,  the  following properties of  a  Killing vector field are  equivalent:
\begin{itemize}
\item[1)]  The  field  $X$  has bounded  length.
\item[2)]  The  field  $X$  is non-trivial and parallel  on $(M,g)$.
\item[3)]  The  field  $X$   is non-trivial parallel  central Jacobi  field  on  $M$.
\item[4)]  $X$  has  constant  length.
\end{itemize}
\end{Cor}
\begin{proof}
 Clearly,  $1)$  $\implies$  $2)$   from   Theorem \ref{const}.
 $2)$   $\implies$   $3)$ follows  from  Corollary  \ref{Jac}  and
 trivially, $3)$    $\implies$  $4)$   $\implies$  $1)$.
 Thus,  all  the  assertions  $1)$, $2)$,   $3)$  and $4)$  are equivalent.
\end{proof}

\begin{Rem}
 Assertions  $1),  2), 4)$  of  Corollary \ref{cl} are  equivalent, for  a Killing vector field  in any manifold of  non-positive sectional  curvature
\cite{BN.08}.
In  case  of   AHM, we  obtain  the  same conclusion without any  assumption on sectional  curvatures. 
\end{Rem}

\subsection{Non  Existence of Killing  Vector Field  of  Constant Length on AH Manifolds with  \boldmath$h >0$\unboldmath.}   

In  this  subsection,  we  find  estimate  
 on  Ricci  curvature  of AH manifold  $M$ of  $h  >0$.  Then  using  this  estimate  we  show  that  there does not
 exist  any  Killing  vector field  on  $M$  of  constant length.  

\begin{Pro}\label{neg-r}
If $(M,g)$  is  an AH manifold  of constant $h > 0$.
Then   Ricci $(v,v)  \leq  \frac{-h^2}{(n-1)} <  0$,  for  any  $v \in  SM$,  where  Ricci  denotes 
the  Ricci  curvature of  $M$.  
\end{Pro}
\begin{proof}
 Taking  the  trace of \eqref{eq:ricatti},
we obtain that on any asymptotically harmonic manifold $M$, 
\begin{equation}\label{eq:ricci<0}
- Ricci(\gamma_v'(t),\gamma_v'(t) )=\tr (u^{+})^2(t),  \;   \forall  v\in SM.
\end{equation}
By  Cauchy-Schwartz  inequality,\\
$$\tr (u^{+})^2(t)  \geq \frac{(\tr u^{+})^2}{(n-1)} =   \frac{h^2}{(n-1)}.$$
Thus,
$Ricci(v,v) \leq - \frac{h^2}{(n-1)} <  0$.
\end{proof}

We  also  recover  Proposition  \ref{non-pos}.

\begin{Cor}
If  $(M,g)$  is  an AHM, then   Ricci $(v,v)  \leq  0, \forall v \in  SM$. 
\end{Cor}

\begin{Lem}\label{k}
If  $(M,g)$  is  any  Riemannian  manifold  of  negative  Ricci curvature,
then  $M$  has  no  non-trivial Killing  vector  fields  of  constant  length.
\end{Lem}
\begin{proof} 
$(M,g)$  be a   Riemannian  manifold  with  $Ricci_{M}  < 0$.
If  $X$  is  a Killing  vector  field  of  constant  length, then  $f$  as in
Proposition  \ref{delta},  is  a  constant function  and hence  harmonic function.
Therefore,  from Proposition  \ref{delta} (3), 
$\Delta f = 0 =  ||\nabla  X||^2  -  Ric(X,X).$
Therefore,
$$ 0 \leq  ||\nabla  X||^2  = Ric(X,X)  < 0.$$
This  implies  that $X =0$.  Thus,  
 Ricci $(X,X)  =  0$ if and  only if  $X  =0$.
\end{proof}

\begin{Cor}
If $(M,g)$  is an  AH manifold  with  $h > 0$,
then  $M$  has  no  non-trivial Killing  vector  field  of  constant  length.
\end{Cor}
\begin{proof}
If  $(M,g)$  is  an  AH manifold of constant $h > 0$,
then  by  Proposition \ref{neg-r},   $Ricci_{M}  < 0$. Hence,  from  Lemma
\ref{k},  the conclusion  follows.  
\end{proof}

\section{ First  order   flatness  of  AHM }

In  case  of   harmonic  manifolds  with  polynomial  volume  growth,
the  vector  spaces   $V =  span\{b_{v} | v  \in T_{p}M \}$  and 
 $W =  span\{b_{v}^2 | v \in T_{p}M \}$  are finite  dimensional;
 where  for  a non-unit  tangent  vector  $v  \in  T_{p}M$,  the  corresponding  Busemann 
function is defined by  $ b_{v}^{+} =  b_{v}  =  ||v||  b_{{v}/||v||}$.  
The   proof   of  flatness  of   HM  relies  heavily   on 
finite  dimensionality  of  $W$  (\cite{RS.02b}).
We  term this  as  {\it second  order  flatness}  of  $M$.   
However,  it  turns  out  that   HM  is  flat  if  and  only  if   
 dim $V =  span\{b_{v} | v  \in T_{p}M \}  =  n = $ dim  $M$.
  This  is  the  most  natural  and  the  strongest  criterion  for  flatness 
 of  HM.  We  term this  as  {\it first  order  flatness}  of  complete,  simply connected  Riemannian  manifold  without 
conjugate  points.   \\ 
 \indent
 In this  section,  we  prove the first  order flatness  of  AHM; viz.,
 that  dim $V =  span\{b_{v} | v  \in T_{p}M \}  =  n$,  consequently,   
 an  AHM is flat.   In fact,  we  prove this  type of  general  flatness  result  
 if  dim $V = n$,  in case  of  complete,  simply connected  Riemannian  manifold  without 
conjugate  points;  using  a  theorem  of   Myers   and  Steenrod.

\begin{Thm} (Myers-Steenrod) \label{MS}
If $\phi: M \rightarrow M$  is  a distance-preserving and 
surjective, then it is an isometry (in particular, it is smooth).
\end{Thm}

\begin{Thm}\label{bus}
If $(M^n,g)$   is a complete,  simply connected  Riemannian  manifold  without 
conjugate  points,  such  that  the  vector  space  $V  =  span\{b_{v} | v  \in T_{p}M \}$  is  $n$-dimensional,  then
 the  map   $F  :  M  \rightarrow  {\R}^n$ defined by
$$F(x)  =  (b_{e_1}(x),  b_{e_{2}}(x),  \cdots, b_{e_{n}}(x)),$$
is  a  $C^1$ isometry. Therefore, $F$ is an  isometry  and  hence 
$F$  is  smooth.  Consequently, $M$  is  flat. 
\end{Thm}
\begin{proof}
If  $(M^n,g)$   is a complete,  simply connected  Riemannian  manifold  without 
conjugate  points,   then  $b_{v}   \in  C^{1}(M)$   \cite{E.77}. 
Define,   
$$G  :   V  \rightarrow  T_{p}M  \;\;\;\mbox{by}\;\;    G(b)  =  \nabla b(p).  $$
Note  that   $G$  is  surjective   as   $\nabla b_{v}(p)  =  -v$. 
Since  both the  vector  spaces   $V$  and $T_{p}M$  have  same  dimension,
 $G$  is  $1-1$  too. \\
  Hence,  
 $G(b_{v+w})  =  - (v+w)  =    G(b_{v})  +    G(b_{w}).$   Thus,
 \begin{eqnarray}\label{s&p}
 b_{v+w}   =  b_{v}  +   b_{w}, \;\;\  b_{av}   = a b_{v} \; \mbox{for} \; a \in  \R.
 \end{eqnarray}
 Observe  that  for  any  $v  \in  TM,$  we  have  
$b_{v}  =  ||v||  b_{{v}/||v||}$,  thus  $||\nabla b_{v}||  =  ||v|| ||\nabla b_{{v}/||v||}||  =  ||v||.$   
 Therefore,  we  obtain,
  \begin{eqnarray*}
  || v_{1}+v_{2}||^2 & = & \left<\nabla b_{v_{1} + v_{2}}  (q),   \nabla b_{v_{1} + v_{2}}  (q)  \right> \\
 &  =  &   \left<\nabla b_{v_{1}}(q) , \nabla b_{v_{1}}(q) \right>  +  \left<\nabla b_{v_{2}}(q) , \nabla b_{v_{2}}(q) \right>   +
 2   \left<\nabla b_{v_{1}}(q) , \nabla b_{v_{2}}(q) \right>   \\
 & = &  ||v_{1}||^2  +   ||v_{2}||^2  +  2  \left<\nabla b_{v_{1}}(q) , \nabla b_{v_{2}}(q) \right>.  
   \end{eqnarray*}
 Consequently,  
 \begin{eqnarray}
 \left<\nabla b_{v_{1}}(q) , \nabla b_{v_{1}}(q) \right>   =     \left<v_{1} , {v_{2}}\right>   
  \end{eqnarray} 
 Thus,   $\{\nabla b_{e_{i}} (q) \},  i  =  1,2, \cdots, n$,  forms  an  orthonormal  basis  of   $T_{q}M$ for  all  $q  \in  M$.  \\
 If $F  :  M  \rightarrow  {\R}^n$ is defined by  $F(x)  =  (b_{e_1}(x),  b_{e_{2}}(x),  \cdots, b_{e_{n}}(x)), \;   \forall  q  \in M.$
 Then,  $F$  is  $C^1$  and  $DF(q)  :  T_{q}M   \rightarrow  {\R}^n$  is  given by
$$  DF(q)(w)  =  ( \left< \nabla b_{e_1}(q),  w \right>,  \cdots,      \left< \nabla b_{e_n}(q),  w\right>). $$
 This  implies  that 
 $$||DF(q)(w)||^2 =   \sum_{i=1}^{n}   \left<\nabla b_{e_i}(q),    w \right>  =  ||w||^2.$$    
 Therefore,
  $DF(q)  :  T_{q}M   \rightarrow  {\R}^n$  is a linear $C^1$ isometry.  
 We  obtain, $F$ is  an  homoeomorphism  which  maps  geodesics  to  geodesics,
 We  conclude   that  $F$  is  distance  preserving  and  we  have  from  Theorem  \ref{MS}, that 
 $F$  is a smooth  isometry and hence  $M$  is  flat.
  \end{proof}

\begin{Cor}
If $(M^n,g)$   is a complete,  simply connected  Riemannian  manifold  without 
conjugate  points.  Let  $\{{e_i}\}$  be  an orthonormal  basis  of  $T_{p}M$  and let  $\{b_{e_{i}}\} $ 
 be  the corresponding  Busemann  functions  on $M$.   Then   the  following  are equivalent:
\begin{itemize}
\item[(i)] The  vector  space  $V =  span\{b_{v} | v \in T_{p}M \}$
is  finite  dimensional  and  dim  $V  = $  dim $M = n$.   
\item[(ii)]
$F  :  M  \rightarrow  {\R}^n$ defined by
$F(x)  =  (b_{e_1}(x),  b_{e_{2}}(x),  \cdots, b_{e_{n}}(x)),$
is  an  isometry   and  therefore, $M$  is  flat. 
\item[(iii)]  $\{\nabla b_{e_i}(p) \}$  is  a  global  parallel  orthonormal basis  of  $T_{p}M$  
for  any  $p \in  M$.   Thus,  $M$  is  a  parallizable  manifold. 
\end{itemize}
\end{Cor}

\begin{Thm}\label{bus-par}
If in  an  AHM, $X$  is a  parallel   vector  field  with  $||X|| = 1$,  then  $X =  \nabla  b_{v}$, 
for  some  $v \in  SM$.  Consequently,  a parallel  Killing  vector  field  in  AHM  with  $||X|| = 1$  is  of the 
form   $\nabla  b_{v}$.
\end{Thm}
\begin{proof}
It is  well  known  that   a  parallel  vector  field  in a  simply  connected  
manifold  is  the  gradient field for a  distance  function   (cf.  \cite{P.98},  pg.  $192$).
We  obtain a  $f  \in  C^{\infty}(M)$  such  that  $||\nabla f|| = 1$
and  $X  =  \nabla f$.   Therefore,   integral  curves  of  $X$  are unit speed  geodesics  of $M$
and level  sets  of  $f$  are  parallel  family  of  hypersurfaces  in  $M$  (refer  \cite{RS.03}
and  \cite{P.98}).  We  have  that every  geodesic  of an  AHM $M$  is a line.  Let  $\gamma_{v}$  be  a  geodesic  of  $M$
with  $\gamma_{v} (0) = p$, then   $X(\gamma_{v}(t))  =  \gamma_{v}'(t) =  \nabla{f}(\gamma_{v}(t)).$
We  may  assume  that  $f(p) =0$. 
Thus,  $\nabla f(p) = \gamma_{v}'(0) = v$  implies  that $f(\gamma_{v}(t)) = -t$.  If  $x \in f^{-1}(-c)$, then
 $d(x,\gamma_{v}(t))  \geq  d( f^{-1}(-t), f^{-1}(-c)) = |-t + c| = t-c $.  Equivalently,   $b_{v}^{+}(x) \geq f(x)$.
 But,  as  $\nabla X  =  {\nabla}^2 f = 0 $,  $f$  is  a  harmonic  function. 
 Therefore,  $b_{v}^{+}(x) - f(x)$  is  a nonnegative  harmonic  function  which  attains  its  minimum  at  $p$
 and  hence  must  be  a  constant  function.   Therefore,   $f =  b_{v}^{+}$.  
 \end{proof}

Now  we  can  prove  our  main  Theorem  \ref{f}  as  stated  in the introduction. \\     

Theorem 1.5 : $(M^n,g)$ be  an AHM  with $\{{e_i}\}$  an orthonormal  basis  of  $T_{p}M$  and   $\{b_{e_{i}}\}$,
 the corresponding  Busemann  functions  on $M$.   Then,
 \begin{itemize}
\item[(1)] The  vector  space  $V =  span\{b_{v} | v \in T_{p}M \}$
is  finite  dimensional  and  dim  $V  = $  dim $M = n$.
\item[(2)]  $\{\nabla b_{e_i}(p) \}$  is  a  global  parallel  orthonormal basis  of  $T_{p}M$  
for  any  $p \in  M$.   Thus,  $M$  is  a  parallizable  manifold.\\
And
\item[(3)]
$F  :  M  \rightarrow  {\R}^n$ defined by
$F(x)  =  (b_{e_1}(x),  b_{e_{2}}(x),  \cdots, b_{e_{n}}(x)),$
is  an  isometry   and  therefore, $M$  is  flat. 
\end{itemize}
  \begin{proof}
If $(M^n,g)$  is an AHM, then  by  Theorem  \ref{bus-par}
there  exists  a  parallel   Killing  vector  field $X$ of the  form   $X =  \nabla  b_{v}$.
Since,  $X$  is  parallel,  the  vector  distribution  orthogonal  to  $X$  is  also  parallel and
involute  on  $M$.  Thus,  by  Theorem \ref{bus-par}, if  $\{{e_i}\}$  is  an orthonormal  basis  of  $T_{p}M$  for  any  $p  \in M$,
then  $\{\nabla b_{e_{i}} \}$  is  a  global  parallel  basis  of  $T_{p}M$  for  any  $p \in  M$.\\
Now  consider the vector  space  $V =  span\{b_{v} | v \in T_{p}M \}$.
Let  $b  =  \sum_{\alpha}  a_{\alpha}  b_{v_{\alpha}} $. But,  $\nabla b  =   \sum_{i=1}^{n}  a_{i}  \nabla  b_{e_i}  $.
This  implies  that   $a_{\alpha}  =  a_{i}$  and   $\nabla  b_{v_{\alpha}}  =    \nabla  b_{e_i}$.
We obtain,   $V =  span\{b_{v} | v \in T_{p}M \}  =   span \{b_{e_{i}} \},$  is  of  dimension $n$.
We conclude from  Theorem  \ref{bus}  that  $M$  is  flat. 
\end{proof}

Second  order  flatness  of  HM  is  already known  from \cite{RS.02b}.  Now  from  the  proof   of  Theorem  \ref{f},
we  also have first  order  flatness  of  HM,  as  per  our  expectations.

\begin{Cor}
If $(M^n,g)$  is an HM, then  the  vector  space  $V =  span\{b_{v} | v \in T_{p}M \}$
is  finite  dimensional  and  dim  $V  = $  dim $M = n$  and  hence,  $M$  is  flat
of order one.
\end{Cor}

\subsection{Comparison  of  first  order and  second order flatness  of  HM}
 In this  subsection,  we  sketch the  proof of second order flatness  of  HM
  of  \cite{RS.02b}  and observe the  merits  of  first  order flatness  of  AHM,
  described  in  this  paper. \\
  $(M,g)$  be  a  complete  manifold.   Fix $p  \in  M$,  $r(x) =  r(p,x)$,
 denotes  distance of  $x \in  M$ from  fixed  point.  ${\cW \cM}_R(\lambda, b)$,
 denotes  the  weak mean value  inequality  as  defined  in  $\S  3$.
  \noindent
 ${\cH}_l(M)$ be the linear space of all harmonic
functions $f$ on $M$ with a polynomial growth of order at
most $l$.
$$ {\cH}_l(M) = \left\{f:M \rightarrow {\R},\;\Delta f 
=  0\;\;\mbox{and}
\;\;\left|f(x)\right| \leq O\left( r^l(x) \right)\right\}.$$

 Theorem 4.2 of  Li -Wang  \cite{LW.99}  is  important
 in  our  discussion.
 
\begin{Thm}\label{fd}   \cite{LW.99}  
  $(M,g)$  be  a  complete  manifold  satisfying  the  weak mean value inequality  
${\cW \cM}_R(\lambda, b)$.  Suppose  that  the  volume growth  of $M$  satisfies
$V_{p}(r)  = O (r^{\nu})$  for  some  point $p \in  M$   and  any  $\nu  > 0$.
 Then   ${\cH}_l(M)$  is  finite dimensional   for  all  $l  \geq  0$  and
dim ${\cH}_l(M) \leq  \lambda (2b+1)^{(2l+\nu)}.$ 
\end{Thm}

\begin{Cor}\label{sqo}
Under  the  conditions  of  Theorem \ref{fd},
the  vector  space  
\begin{eqnarray*}
{\cF} = \left\{ f:M \rightarrow {\R},\;\Delta f =
c\;\;\mbox{and} \;\;\left| f(x) \right|
\leq O \left( r^2(x) \right) \right\}.
\end{eqnarray*}  
 is also a finite dimensional space  and   dim ${\cF} = 1+ $ dim   ${\cH}_2(M).$  
 \end{Cor}

It is well known that all
harmonic manifolds satisfy Mean Value Inequality $ {\cM}_R(\lambda)$ with 
$\lambda = 1$ for all $R$. See  \cite{W.50}.
Thus,  they  also  satisfy   ${\cW \cM}_R(\lambda, b)$,  for  $\lambda = 1$ 
and  any  $b >1$.  

\begin{Cor}\label{est}
 $(M^n,g)$  be  a  harmonic manifold of  polynomial volume growth,
$V_{p}(r)  = O (r^{\nu})$. Then  $V =  span\{b_{v} | v  \in T_{p}M \}$   and $W =  span\{b_{v}^2 | v \in T_{p}M \}$  
are finite  dimensional vector  spaces  of  dimension  $\leq  3^{(2+n)}$  and  $\leq  1+  3^{(4+n)}$, 
respectively.  
 \end{Cor}
\begin{proof}      
 If  $(M,g)$ is  a  harmonic manifold of  polynomial volume growth,
$V_{p}(r)  = O (r^{\nu})$  implies  that  $V_{p}(r)  = O (r^n)$  and
 $\Delta  b_{v}  = 0$,  for  all  $v  \in  SM$.   Note  that as $b_{v}$  is a  Lipschitz  function, 
$|b_{v}(x) | \leq  O(r(x))$ and hence  $V =  span\{b_{v} | v \in S_{p}M \}$  is  finite  dimensional
(Theorem  \ref{fd})  and  dim $V \leq  (2b+1)^{(2+\nu)}=  3^{(2+n)} >>>n.$ \\  Also,  
$\Delta  b_{v}^2  =  2$  for all $v \in  SM$  and $|b_{v}^2(x) | \leq  O(r^2(x))$.
Consequently,  $W =  span\{b_{v}^2 | v \in T_{p}M \}$  (Corollary  \ref{sqo})
is  also  finite  dimensional vector  space and  dim  $W  \leq  1+ (2b+1)^{(4+\nu)} = 1+  3^{(4+n)}.$ 
\end{proof}

We  shortly  describe the proof  of  theorem  in  \cite{RS.02b},  that  harmonic  manifolds  with
polynomial volume growth  are flat.

\begin{Thm}\label{sketch}
Harmonic  manifolds  with
polynomial  volume  growth  are flat
\cite{RS.02b}.
\end{Thm}
\begin{proof}      
The original proof of  Theorem \ref{flat1} given in \cite{RS.02b}  is based 
 on  the idea of the Szabo's proof of the Lichnerowicz's conjecture  in compact case.
 In the proof,  ${b_v}^2$ was averaged  (idea which can be employed only for harmonic manifolds),
  and a parallel displaced family, $g_{\gamma}$, of real valued functions on $\R$, for every geodesic $\gamma$  was  obtained.
 As  span$\{b_v ^{2}\}$  is a finite dimensional vector space  (Corollary  \ref{est}), 
 span $\{g_{\gamma} \}$  is also a finite dimensional vector space  and  therefore, the generator function $g$ is a
  trigonometric polynomial.  By using properties of $g$, $g$ was  written in the simpler form. 
  Then another family of radial functions ${\mu}_{\gamma}$  was introduced. The  generator function, $\mu$  was
   obtained by generalizing co-ordinate function  $r \cos \theta$ on a harmonic manifold. 
   Then using properties of $\mu$,  the two families were  related. It  was observed that $g$ and $\mu$ are almost periodic functions. 
   Finally, using the Characteristic Property of an almost periodic function, it  was proved that $M$ is Ricci flat  and  hence  flat.
 \end{proof}

\indent Conclusion : 
\begin{itemize}
\item[1)] Note  that above sketch of Theorem  \ref{sketch} of  \cite{RS.02b} shows  that
the  flatness  of  HM  was  proven by  using finite  dimensionality of, $W = span b_{v}^2$ completely.
\item[2)] The  sketch  also  shows  that  the  proof uses  harmonicity  of  
$M$  heavily,  while  the  proof obtained in this  paper  for  AHM  is wider  and  strongest in comparison. 
\item[3)] Also  from  Corollary \ref{est},  it  is  not  clear  whether  dim  $V  \leq  n$. In fact, the  bound  $3^{(2+n)} >>> n$  and  dimension estimate  of  \cite{LW.99}
is  not  optimal, in this  case.
\end{itemize}

\section{Consequences  of  Flatness of  AHM}

In  this  section,  we  describe  importantance  
 of  our  main  result  that  an AHM   is  flat.   We  need  the  following  definitions.

\begin{Def}\label{ncom}
 If  $(M,g)$  is a connected,  non-compact  Riemannian manifold,  then  the volume  entropy  
 $h_{vol}(M)$  of  $M$  is  defined as  :
 \begin{eqnarray}\label{ve}
h_{vol} (M,g)  =  \lim_{R  \rightarrow  \infty} \sup \frac{\log (\Vol(B_{p}(R)))}{R},
\end{eqnarray}
 \end{Def}
where  $B_{r}(p)  \subset  M$ is the open ball of radius $R$ around $p \in  M$.\\
 
 \noindent
 Note  that  (\ref{ve})  doesn't  depend  on the choice  of  reference  point  $p$
 and  $h_{vol} (M,g)$  is  therefore well  defined.

 \begin{Cor}\label{gtype1}
If $(M,g)$  is an  AH manifold  of  constant  $h \geq 0$,
then  $M$  has either  polynomial  volume  growth  or  exponential  volume growth.
\end{Cor}
\begin{proof}
 $(M,g)$  be an  AH manifold  of  constant  $h \geq 0$.
If  $h =0$,  then  from our main Theorem  \ref{f},  M  is  flat  and hence of  polynomial  volume  growth.
If  $h > 0$,  then  as   $\displaystyle \lim_{r  \rightarrow  \infty}   \frac{ \Theta_{q}'(r, v)}{\Theta_{q}(r, v)}  = h  > 0,$
for  any  $q \in  M$ and $v \in S_{q}M$, volume  growth  is  exponential.
\end{proof}
\noindent
Thus, we  have  generalized   the  result  of  \cite{N.05},  Corollary  $1.4$  stated in $\S 1$: \\

\noindent
Corollary 1.4: If $(M,g)$  is a harmonic  manifold  of  constant  $h \geq 0$,
then  $M$  has either  polynomial  volume  growth  or  exponential  volume growth. \\

\noindent
We  have  obtained:

\begin{Thm}\label{growth}
If $(M,g)$  is an  AH manifold  of  constant  $h \geq 0$,
then  the following propeties  are equivalent:
\begin{itemize}
\item[1)]  $M$  has  subexponential  volume  growth.
\item[2)] $h = 0$.
\item[3)]  $M$  is  flat.
\item[4)]  $M$  is  of  polynomial  volume  growth.
\item[5)]  Volume  entropy of  $M$,  $h_{vol}(M) =  0= h $.
\item[6)]  $M$ has  rank  $n$.
\end{itemize}
\end{Thm}

In  \cite{ISS.14}  among AH Hadamard manifolds, a rigidity theorem
 with respect to volume entropy for  real, complex and  quaternionic hyperbolic   spaces were obtained.  
In particular, the  following  result was  proved:

\begin{Thm}\label{ent}
 $(M,g)$ be an $n$-dimensional Hadamard manifold of Ricci curvature Ricci $\geq  -(n-1)$. If
$(M,g)$ is AH, then  $h_{vol}  =  (n-1)$  and equality holds if and only if
 $(M,g)$ is isometric to the real hyperbolic space ${\Hy}^n$ of constant curvature $-1$.
\end{Thm}

 The  following  result  was  proved  in  \cite{KP.05}:

\begin{Thm}\label{expo}
Let  $(M,g)$  be an  AH manifold  of  constant  $h \geq 0$, such  that  
$||R||  \leq  R_{0}$   and  $||\nabla R||  \leq  R_{0}'$  with  suitable  constants
$R_{0},  R_{0}' > 0$.  Then  the following propeties  are equivalent:
\begin{itemize}
\item[1)]  $M$ has  rank  $1$.
\item[2)]  $M$  has  Anosov  geodesic  flow  ${\phi}^t   :  SM  \rightarrow  SM$.
\item[3)]  $M$  is  Gromov hyperbolic.
\item[4)]  $M$  has  purely  exponential  volume  growth  with  growth  rate\\ $h_{vol}  = h > 0$.
\end{itemize}
\end{Thm}
 
 \begin{Rem}
\begin{itemize}
\item[1)] Our Theorem  \ref{growth} is complementary  to  Theorem  \ref{ent}   and  Theorem \ref{expo}. 
\item[2)] In  view  of  Corollary  \ref{gtype1}   (Corollary  1.4),  harmonic  manifolds  and AH  manifolds share  the common property  viz.  that  they  have  the  same volume  growth. This  is  an important observation towards proving  that  class  of  AH  and  harmonic manifolds coincide.
\end{itemize}
\end{Rem}

\subsection{Final conclusion}
We showed in $\S 3$ that the Strong Liouville Type Property holds for more general
type of spaces viz., for non-compact, complete, connected manifold of infinite injectivity radius and  subexponential volume growth. Similarly, the results in the other sections of this paper, like the existence of Killing vector fields on AHM, can be
proved in the more generality. Hence, we can conclude first order flatness even for these general type of spaces.\\

\noindent
{Theorem $1.6$ :}\\
If $(M,g)$ is a non-compact, complete, connected manifold of infinite injectivity radius and  subexponential volume growth, then $(M,g)$ is first order flat.

 \section{Appendix \\  Asymptotically  Harmonic  Manifolds  With  Minimal  Horospheres  Admitting Compact  Quotient}  
 
 In  this  Appendix,  we  prove:
 \begin{Thm}\label{co}
 If  $(M,g)$  is  an an  AHM, admitting  compact quotient,  then  $M$  is  flat  
 \end{Thm}
 \noindent
Note  that:\\
 (1) Theorem  \ref{co}  should be known, but the author couldn't 
find  the reference,where  it is  explicitly proved and  hence she includes it  here. \\
(2)  In  view  of  Theorem  \ref{co},  the present  paper  is  devoted  to  proving flatness  of  non-compact  AHM.\\

  We  recall  here  definition  of  complete   AH manifold,  that  of volume
 growth  and volume  growth  entropy. 
  
 \begin{Def}
  \begin{itemize}
 \noindent
 \item[1)] $(M,g)$  be  a  complete Riemannian  manifold without  conjugate points  and $(\tilde M, g)$  be  its   universal  Riemannian  covering
 space.  Then  $(M,g)$  is  called  AH,  if  there exists a  constant  $h \geq 0$  such that  the  mean curvature of
 all  horospheres in  $\tilde M$ is  $h$.  
 \item[2)]  Let  $(M,g)$ and  $(\tilde M, g)$  be as in 1)  and
  choose  $p  \in  \tilde M$.  Then the  volume  entropy  (or  asymptotic  volume  growth)  $h_{v} = h_{v} (M,g)$  
 is  defined  as  :
 $$ h_{v} = h_{v} (M,g)  =  \lim_{R  \rightarrow  \infty}  \frac{\log (\Vol_{\tilde M} (B_{p}(R)))}{R}.$$      
 \end{itemize}
 \end{Def} 
 Manning  \cite{M.79}  showed that the limit above always exists and is independent of  $p \in  \tilde M$. \\
 \noindent
 In  \cite{Z.12}  the  following  Proposition was proved. 
 
 \begin{Pro}\label{foc}
If $(M,g)$ is a compact AHM  without  focal points, then $M$  is  flat.
 \end{Pro}
 
 We  prove the  above mentioned  result  without  no  focal point  condition,  mainly  by  using  the 
 following  result  of  \cite{Z.12}  and  Tits  Theorem.
 
 \begin{Pro}\label{equ}
 If  $(M,g)$  is a  compact   AH manifold  with  ${\tilde M}$
 as  universal  covering  space  of $M$  and  $\Delta  b_{v}  =  h$, for  all 
 $v \in  S  {\tilde M}$,  then  $h =   h_{v}$. 
  \end{Pro}
 
 \begin{Thm}\label{Tit}
 {\bf  Tits  Theorem:}   A  finitely  generated  subgroup  of  a connected  Lie  group has either 
 exponential  growth  or  is  almost  nilpotent  and  hence has  polynomial  growth.
 \end{Thm}
 
 \begin{Thm}\label{co1}
If $(M,g)$  is  compact  AHM,
then  $(M,g)$  is  flat.
\end{Thm}
\begin{proof}
If  $\tilde  M$  is  the  universal covering  space  of $M$, then  $\Delta  b_{v}  \equiv  0,$
for  all  $v  \in  S {\tilde M}$.   And hence the volume growth  entropy  $h_{v} = 0$,  by  Proposition \ref{equ}.  
We obtain,  volume  growth  of  $\tilde M$  is subexponential  and consequently,
  $\pi_{1}(M)$  is of  subexponential  growth.
 But  by  Tits'  Theorem,   $\pi_{1}(M)$   has  polynomial  growth,  
 as  it  is  a  subgroup  of  connected  Lie  group  Isom$(\tilde M)$,
 of   subexponential  growth.   By  \cite{L.05} compact  Riemannian manifold 
  without  conjugate  points and  with polynomial  growth  fundamental group   are  flat.
  We  conclude that  $(M,g)$  is  flat.
  \end{proof}

\noindent 
 As   any  harmonic  manifold is   AH we  conclude:
 
\begin{Cor}
If $(M,g)$  be a  compact HM,  then  $M$  is  flat.  
\end{Cor}

\end{document}